\documentclass{article}
\usepackage{amsmath,amssymb,amsthm,thmtools,enumitem,bm,graphicx,xcolor,microtype}
\usepackage[T1]{fontenc}

\usepackage[style = ext-numeric, sorting = nyt, giveninits = false, maxbibnames = 8, minbibnames = 3, backref = true]{biblatex}
\renewbibmacro{in:}{}
\DeclareFieldFormat{titlecase}{\MakeSentenceCase*{#1}}
\DeclareFieldFormat{titlecase:journaltitle}{{#1}}

\usepackage[unicode]{hyperref}

\theoremstyle{definition}
\newtheorem{dfn}{Definition}[section]
\newtheorem{ntt}[dfn]{Notation}
\newtheorem{ex}[dfn]{Example}
\newtheorem{case}{Case}

\theoremstyle{remark}
\newtheorem{rmk}[dfn]{Remark}

\theoremstyle{plain}
\newtheorem{thm}[dfn]{Theorem}
\newtheorem*{thm*}{Theorem}
\newtheorem{prop}[dfn]{Proposition}
\newtheorem{lem}[dfn]{Lemma}

\newtheorem{fct}[dfn]{Fact}

\newtheorem{qst}[dfn]{Question}
\newtheorem{clm}{Claim}

\declaretheoremstyle[headfont = \normalfont\itshape, qed = $\blacksquare_{\mathrm{Claim}}$]{prooc}
\declaretheorem[style = prooc, name = Proof of Claim, numbered = no]{prooc}

\def\sectionautorefname~#1\null{Chapter~#1\null}
\def\subsectionautorefname~#1\null{Section~#1\null}

\title{Distal expansions of the integers and the $p$-adic fields}
\author{Koki Okura\thanks{Doctoral Program in Mathematics, University of Tsukuba, Ibaraki 305-8571, Japan\\E-mail: \href{\detokenize{mailto:k_okura@math.tsukuba.ac.jp}}{\detokenize{k_okura@math.tsukuba.ac.jp}}}}
\date{}

\begin{document}

\maketitle

\begin{abstract}
    This paper investigates expansions of distal structures by a unary subset that arises as the image of a projection map.
    We first provide a sufficient condition for such an expansion to remain distal.
    Based on this criterion, we establish the distality of three kinds of expansions involving the integers or the $p$-adic fields.
    Let $R$ be an almost sparse sequence.
    We prove that $(\mathbb{Z};<,+,R)$ is distal, thereby answering a question posed by Tong.
    Furthermore, we show the distality of $(\mathbb{Q}_p;+,\cdot,p^{\mathbb{Z}})$ and $(\mathbb{Q}_p;+,\cdot,p^{\mathbb{Z}},p^R)$.The latter provides an example of a NIP expansion of the $p$-adic field without the rationality of the Poincar\'e series.
\end{abstract}

\tableofcontents

\section{Introduction}
\label{chp:Intro}

In 2013, Simon \cite{Simon2013} introduced \emph{distal theories} as a special subclass of NIP theories, intended to characterize those NIP theories that are `purely unstable.'
Distality not only provides a finer classification within NIP but also implies desirable combinatorial properties such as the strong Erd\H{o}s-Hajnal property and upper bounds for Zarankiewicz's problem \cite{Chernikov2018,Chernikov2020}.

Numerous theories have been identified as distal.
Typical examples include o-minimal and P-minimal theories; in particular, the real field and the $p$-adic fields are distal.
Another example is the ordered additive group of integers.
An AKE-style characterization of distal henselian valued fields was given in \cite{Aschenbrenner2022}.

A subject of ongoing interest is determining which expansions of distal structures remain distal.
For instance, Hieronymi and Nell \cite{Hieronymi2017} investigated expansions of o-minimal ordered groups, proving the distality of expansions of a real field by a cyclic multiplicative subgroup of its positive part.
In particular, $(\mathbb{R};+,\cdot,2^\mathbb{Z})$ is distal, where $2^\mathbb{Z} = \left\{ 2^z \mid z \in \mathbb{Z} \right\}$.
In a different direction, Tong \cite{Tong2025} proved that expansions of $(\mathbb{Z};<,+)$ by a `sparse sequence' are distal, provided the sequence is `congruence-periodic.'

The main result of this paper is the following:
\begin{thm*}
    Let $R$ be an almost sparse sequence of integers.
    Let $p^\mathbb{Z} = \left\{ p^z \mid z \in \mathbb{Z} \right\}$ and $p^R = \left\{ p^r \mid r \in R \right\}$.
    Then, the following structures are distal:
    \begin{itemize}
        \item $(\mathbb{Z};<,+,R)$
        \item $(\mathbb{Q}_p;+,\cdot,p^\mathbb{Z})$
        \item $(\mathbb{Q}_p;+,\cdot,p^\mathbb{Z},p^R)$
    \end{itemize}
\end{thm*}

Notably, we do not require $R$ to be congruence-periodic.
Consequently, the first result provides an affirmative answer to the question of whether congruence periodicity can be omitted, as posed by Tong \cite[Question 1.4]{Tong2025}.
The second result is analogous to the distality of $(\mathbb{R};+,\cdot,2^\mathbb{Z})$ and strengthens the NIP of $(\mathbb{Q}_p;+,\cdot,p^\mathbb{Z})$, which is an immediate consequence of Mariaule's result \cite[Corollary 6.5]{Mariaule2018}.
The third structure combines the previous two; as it was not previously known to even have NIP, it serves as the first example of an NIP expansion of the $p$-adic field that lacks the rationality of the Poincar\'e series.

To prove these results, we first establish a distality criterion for an expansion of a distal structure $\mathcal{N}$ by a unary subset $R$ that is the image of a projection map.
Briefly stated, $(\mathcal{N},R)$ is distal if:
\begin{enumerate}
    \item Its elementary behavior is determined by that of $\mathcal{N}$ and a distal structure $\mathcal{R}$ on $R$.
    \item It admits `variable separation.'
\end{enumerate}
This criterion is presented in \autoref{chp:Ds}, following a review of the definition and basic properties of distality.
Similar criteria have been proposed by Hieronymi and Nell \cite[Theorem 2.1]{Hieronymi2017} and Aschenbrenner, Chernikov, Gehret, and Ziegler \cite[Proposition 7.1]{Aschenbrenner2022}.
We expect our criterion to be applicable to various other expansions.
For instance, while Hieronymi and Nell (ibid.) established the distality of pairs of real closed fields with a standard part map, our criterion might be used to obtain an analogous result for $p$-adically closed fields.

To verify the first condition of the criterion, we employ back-and-forth arguments.
\autoref{chp:ZR} is devoted to the structure $(\mathbb{Z};<,+,R)$.
After examining the basic properties of almost sparse sequences, we establish the necessary back-and-forth and variable separation results, which immediately imply distality via our criterion.
The structures $(\mathbb{Q}_p;+,\cdot,p^\mathbb{Z})$ and $(\mathbb{Q}_p;+,\cdot,p^\mathbb{Z},p^R)$ are treated in \autoref{chp:QR}.
Although quantifier elimination for $(\mathbb{Q}_p;+,\cdot,p^\mathbb{Z})$ was established by Mariaule \cite{Mariaule2017}, we perform a back-and-forth analysis for $(\mathbb{Q}_p;+,\cdot,p^\mathbb{Z},p^R)$.
Subsequently, we establish a common form of variable separation that applies to both structures.
Their distality is again an immediate consequence of our criterion.
Finally, we introduce Poincar\'e series and demonstrate that $(\mathbb{Q}_p;+,\cdot,p^\mathbb{Z},p^R)$ does not have the rationality of the Poincar\'e series.

\section*{Acknowledgement}

This work was supported by JST SPRING, Grant Number JPMJSP2124 and the Research Institute for Mathematical Sciences, an international Joint Usage/Research Center located in Kyoto University.
The author is deeply grateful to his supervisor, Kota Takeuchi, for his thorough guidance throughout the preparation of this paper.

\subsection{Notation}
\label{sec:IntroNtt}

We write $\overline{a} \in X$ to indicate that each of its component belongs to $X$.
We sometimes denote the concatenation of tuples $\overline{a}$ and $\overline{b}$ or the union of sets $A$ and $B$ by juxtaposition notation $\overline{a}\overline{b}$ or $AB$.
For an $L$-theory $T$, $|T|$ denotes the number of $L$-formulas, i.e., $|T| = |L| + \aleph_0$.
For a subset $X$ of a structure and a set of functions $\mathcal{F}$, we let $\langle X \rangle_\mathcal{F}$ denote the subset generated by $X$ with respect to $\mathcal{F}$:
\begin{equation*}
    \langle X \rangle_\mathcal{F} = \left\{ t(\overline{a}) \mid \text{$\overline{a} \in X$ and $t$ is a term involving functions in $\mathcal{F}$}  \right\}.
\end{equation*}

\begin{rmk}\label{rmk:Monster}
    Throughout most of this paper, we work within monster models.
    By a `monster model', we mean a structure $\mathcal{U}$ such that any structure interpretable in $\mathcal{U}$ (including $\mathcal{U}$ itself) is $\kappa$-saturated and strongly $\kappa$-homogeneous for a sufficiently large cardinal $\kappa$.
    Hodges \cite[Chapter 10]{Hodges1993} introduces the notion of $\kappa$-bigness and demonstrates that it implies both $\kappa$-saturation and strong $\kappa$-homogeneity.
    Furthermore, $\kappa$-bigness is preserved under interpretation, and for any $\kappa$, every structure has a $\kappa$-big elementary extension.
    Thus, $\kappa$-big models serve as suitable monster models.
    Unless an ambient model is specified, we take elements and parameter sets from a monster model.
    A subset of a $\kappa$-big monster model is called \emph{small} if its cardinality is less than $\kappa$.
    We assume that all tuples and parameter sets considered in a monster model are small.
\end{rmk}

\section{Distality}
\label{chp:Ds}

\subsection{Generalities on distality}
\label{sec:DsIntro}

Throughout this section, we fix a complete $L$-theory $T$ with infinite models.
We say the \emph{dp-rank} of $T$ is at least $\kappa$ if it has an ICT-pattern of depth $\kappa$.
(We omit the formal definition of an ICT-pattern as it is not required for our arguments.)
If a theory has dp-rank larger than any cardinal, we consider the rank to be $\infty$.
While several equivalent definitions exist, $T$ has \emph{NIP} (the Non-Independence Property) if its dp-rank is less than $\infty$, or equivalently, less than $|T|^+$.
Theories with dp-rank $1$ are called \emph{dp-minimal}.
For a comprehensive treatment of NIP theories, we refer the reader to Simon's textbook \cite{Simon2015}.

Simon \cite{Simon2013} introduced the notion of distality to characterize NIP theories that are `purely unstable.'

\begin{dfn}
    Consider linearly ordered sets $I$ and $J$.
    We denote by $I + J$ the concatenation of $I$ and $J$ (i.e., $i < j$ for any $i \in I$ and $j \in J$).
    A cut $(I,J)$ of $I + J$ is \emph{Dedekind} if $I$ has no maximum element and $J$ has no minimum element.
\end{dfn}

\begin{dfn}\label{dfn:Distality}
    $T$ is \emph{distal} if the following holds:
    For any sequence $(\overline{a_k})_{k\in I+(c)+J}$ and any singleton $b$, if:
    \begin{itemize}
        \item $(I,J)$ is Dedekind;
        \item $(\overline{a_k})_{k\in I+(c)+J}$ is indiscernible;
        \item $(\overline{a_k})_{k\in I+J}$ is indiscernible over $b$;
    \end{itemize}
    then $(\overline{a_k})_{k\in I+(c)+J}$ is indiscernible over $b$.
    A structure is distal if its theory is distal.
\end{dfn}

\begin{rmk}
    \autoref{dfn:Distality} is `external characterization' of distality rather than its original definition (see Simon \cite{Simon2013}).
    Other characterizations, such as those involving distal invariant types, strong honest definitions, or the smoothness of generically stable measures, can be found in Simon \cite{Simon2015}.
\end{rmk}

\begin{rmk}\label{rmk:DistalImplications}
    Distal theories are known to possess NIP.
    See Gehret and Kaplan \cite[Proposition 2.9]{Gehret2020} for a quick proof or Walker \cite[\S6]{Walker2023} for a stronger result.
    Unlike NIP, a reduct of a distal structure is not necessarily distal.
    For example, $(\mathbb{R},<)$ is distal, but its reduct to the empty language is not.
    It is known that dp-minimal structures with no unrealized generically stable type are distal \cite[Corollary 2.30]{Simon2013}.
    Such structures include all dp-minimal linearly ordered structures and the $p$-adic fields.
\end{rmk}

\begin{lem}\label{lem:BasicsofDistality}
    Equivalent definitions of distality are obtained by:
    \begin{enumerate}
        \item Allowing $\overline{a_k}$ to be an infinite tuple.
        \item Replacing `indiscernible' with `indiscernible over $C$' and `indiscernible over $b$' with `indiscernible over $Cb$', where $C$ is an arbitrary parameter set.
        \item Replacing $b$ with an arbitrary parameter set $B$.
        \item Requiring only that $\mathrm{tp}(\overline{a_k} / b) = \mathrm{tp}(\overline{a_c} / b)$ for any $k \in I+J$, rather than the indiscernibility of $(\overline{a_k})_{k\in I+(c)+J}$ over $b$.
        \item Restricting $(I,J)$ to a specific class of Dedekind cuts.
    \end{enumerate}
\end{lem}

\begin{proof}
    (1) Note that in general, if $|\overline{a_k}| = \infty$, then $(\overline{a_k})_{k\in I}$ is indiscernible if and only if for every finite subtuple $\widehat{a_k}$ of $\overline{a_k}$, $(\widehat{a_k})_{k\in I}$ is indiscernible.

    (2) Generally speaking, $(\overline{a_k})_{k\in I}$ is indiscernible over $C$ if and only if $(\overline{a_k}C)_{k\in I}$ is indiscernible, where we regard $\overline{a_k}C$ as an augmented tuple.

    (3) By well-ordering $B$ as $\left\{ b_{\alpha} \mid \alpha < \lambda \right\}$, we can show that $(\overline{a_k})_{k\in I+(c)+J}$ is indiscernible over $\left\{ b_{\beta} \mid \beta < \alpha \right\}$ for any $\alpha \leq \lambda$ by transfinite induction.

    (4)
    Consider a sequence $(\overline{a_k})_{k\in I+(c)+J}$ and a singleton $b$ satisfying the conditions in \autoref{dfn:Distality}.
    To show that $(\overline{a_k})_{k\in I+(c)+J}$ is indiscernible over $b$, it suffices to take arbitrary $m,n \geq 0$, $i_1<\dots<i_m<i_{m+1}$ from $I$, and $j_1<\dots<j_n$ from $J$, and show that:
    \begin{equation*}
        \mathrm{tp}(\overline{a_{i_1}}\dots\overline{a_{i_m}}\overline{a_c}\overline{a_{j_1}}\dots,\overline{a_{j_n}} / b)
        = \mathrm{tp}(\overline{a_{i_1}}\dots\overline{a_{i_m}}\overline{a_{i_{m+1}}}\overline{a_{j_1}}\dots,\overline{a_{j_n}} / b).
    \end{equation*}
    Since $(I,J)$ is Dedekind, we can choose infinite sequences $i_{m+1} = i'_0 < i'_1 < \dots$ from $I$ and $j'_0 > j'_1 > \dots$ from $J$ such that $j_1 > j'_0$.
    Let $I' = (i'_k)_{k \in \omega}$ and $J' = (j'_k)_{k \in \omega^*}$, where $\omega^*$ is the reverse order type of $\omega$.
    Then:
    \begin{itemize}
        \item $(\overline{a_k})_{k\in I'+(c)+J'}$ is indiscernible over $\overline{a_{i_1}}\dots\overline{a_{i_m}}\overline{a_{j_1}}\dots\overline{a_{j_n}}$;
        \item $(\overline{a_k})_{k\in I'+J'}$ is indiscernible over $\overline{a_{i_1}}\dots\overline{a_{i_m}}\overline{a_{j_1}}\dots\overline{a_{j_n}}b$.
    \end{itemize}
    Condition (4) then implies:
    \begin{equation*}
        \mathrm{tp}(\overline{a_c} / \overline{a_{i_1}}\dots\overline{a_{i_m}}\overline{a_{j_1}}\dots\overline{a_{j_n}}b)
        = \mathrm{tp}(\overline{a_{i_{m+1}}} / \overline{a_{i_1}}\dots\overline{a_{i_m}}\overline{a_{j_1}}\dots\overline{a_{j_n}}b),
    \end{equation*}
    which yields the desired equality.

    (5)
    We show that if $T$ is not distal, a witness of non-distality can be found for an arbitrary Dedekind cut $(\widehat{I},\widehat{J})$.
    Suppose that $T$ is not distal.
    By (4), there exist a Dedekind cut $(I,J)$, a sequence $(\overline{a_k})_{k\in I+(c)+J}$, a singleton $b$, and a formula $\varphi(\overline{x},y)$ such that:
    \begin{itemize}
        \item $(\overline{a_k})_{k\in I+(c)+J}$ is indiscernible;
        \item $(\overline{a_k})_{k\in I+J}$ is indiscernible over $b$;
        \item $\models \varphi(\overline{a_k},b)$ for all $k \in I + J$, but $\not\models \varphi(\overline{a_c},b)$.
    \end{itemize}
    By compactness, we can construct a sequence $(\overline{a'_k})_{k\in \widehat{I}+(c)+\widehat{J}}$ satisfying these same conditions for the given cut $(\widehat{I},\widehat{J})$.
\end{proof}

\begin{rmk}\label{rmk:SaturatedIndices}
    Based on \autoref{lem:BasicsofDistality} (5), we shall assume throughout this paper that the index sets $I$ and $J$ are sufficiently saturated dense linear orders without endpoints.
\end{rmk}

\subsection{Criterion for distality}
\label{sec:DsCriterion}

In this section, we provide a sufficient condition for an expansion $(\mathcal{N},R)$ of a structure $\widetilde{\mathcal{N}} = \mathcal{N}$ to be distal, where $R$ is a unary subset of $N$.

The setting is as follows:
Let $\mathcal{R}$ be a structure on $R$ that is $\emptyset$-definable in $\widetilde{\mathcal{N}}$ (that is, all function and relation symbols in $\mathcal{R}$ are $\emptyset$-definable in $\widetilde{\mathcal{N}}$).
Let $\mathcal{F}_N$ (resp. $\mathcal{F}_R$) be a set of functions on $N^n$ (resp. $R^n$) for some $n$ that are $\emptyset$-definable in $\mathcal{N}$ (resp. $\mathcal{R}$).
Let $\lambda:N \to R$ be a function $\emptyset$-definable in $\widetilde{\mathcal{N}}$.
We regard each $h \in \mathcal{F}_R$ as a function on $N^n$ by defining $h(\overline{x}) = h(\lambda(x_1),\dots,\lambda(x_n))$.
Let $\mathcal{F} = \mathcal{F}_N \cup \mathcal{F}_R \cup \left\{ \lambda \right\}$.

Since we work within a monster model of $\widetilde{\mathcal{N}}$ when demonstrating its distality, we assume henceforth that $\widetilde{\mathcal{N}}$ itself is a monster model.

\begin{dfn}\label{dfn:IndisParameter}
    For a sequence $(\overline{a_k})_{k\in I+(c)+J}$, we say $(I',J')$ is a \emph{truncation} of $(I,J)$ if $I'$ is a nonempty final segment of $I$ and $J'$ is a nonempty initial segment of $J$, and write $(I',J') \subset_{\mathrm{t}} (I,J)$.
    A \emph{parameter set outside} $(I',J')$ is a small subset of $\left\{ \overline{a_k} \mid k \in (I \setminus I') \cup (J \setminus J') \right\}$.
\end{dfn}

\begin{dfn}\label{dfn:SeparationofTerms}
    In the above setting, $\widetilde{\mathcal{N}}$ admits \emph{variable separation} if for any sequence $(\overline{a_k})_{k\in I+(c)+J}$ and any singleton $b$ satisfying the assumptions in \autoref{dfn:Distality}, and for any $\mathcal{F}$-term $t(\overline{x},y)$, there exist:
    \begin{itemize}
        \item a truncation $(I',J') \subset_{\mathrm{t}} (I,J)$,
        \item an $\mathcal{F}_N$-term $u(\overline{z},\overline{w})$,
        \item tuples of $\mathcal{F}$-terms $\overline{r}(\overline{x}) = (r_1(\overline{x}),\dots,r_m(\overline{x}))$ and  $\overline{s}(y) = (s_1(y),\dots,s_n(y))$ involving parameters outside $(I',J')$,
    \end{itemize}
    such that for any $k \in I' + (c) + J'$:
    \begin{equation*}
        t(\overline{a_k},b) = u(\overline{r}(\overline{a_k}),\overline{s}(b)),
    \end{equation*}
    and if $t$ is of the form $h(\overline{t})$ for some $h \in \mathcal{F}_R \cup \left\{ \lambda \right\}$, then $\overline{s}(b) \in R$.
\end{dfn}

We can choose a truncation $(I',J')$ which works for all terms $t$ and produce no new parameter.
Specifically:

\begin{lem}\label{lem:StrongerSeparationofTerms}
    Suppose that $\widetilde{\mathcal{N}}$ admits variable separation.
    Take a sequence $(\overline{a_k})_{k\in I+(c)+J}$ and a singleton $b$ satisfying the assumptions in \autoref{dfn:Distality}.
    Then there exist a truncation $(I',J') \subset_{\mathrm{t}} (I,J)$ and a parameter set $P$ outside $(I',J')$ with the following property:
    For any $\mathcal{F}$-terms $t(\overline{x},y)$ involving $P$, there exist an $\mathcal{F}_N$-term $u(\overline{z},\overline{w})$ and tuples of $\mathcal{F}$-terms $\overline{r}(\overline{x})$, $\overline{s}(y)$ involving $P$ such that for any $k \in I' + (c) + J'$:
    \begin{equation*}
        t(\overline{a_k},b) = u(\overline{r}(\overline{a_k}),\overline{s}(b)),
    \end{equation*}
    and if $t$ is of the form $h(\overline{t})$ for some $h \in \mathcal{F}_R \cup \left\{ \lambda \right\}$, then $\overline{s}(b) \in R$.
\end{lem}

\begin{proof}
    Let $(I_0,J_0)$ = $(I,J)$ and $P_0 = \emptyset$.
    We construct a sequence of truncations $(I_0,J_0) \supset_{\mathrm{t}} (I_1,J_1) \supset_{\mathrm{t}} \dots$ and a sequence of parameter sets $P_0 \subseteq P_1 \subseteq \dots$ such that $P_m$ is outside $(I_m,J_m)$.
    Suppose that we have obtained $(I_m,J_m)$ and $P_m$.
    By well-ordering the set of $\mathcal{F}$-terms involving $P_m$ and applying variable separation transfinitely many times, we can construct a truncation $(I_{m+1},J_{m+1}) \subset_{\mathrm{t}} (I_m,J_m)$ and $P_{m+1} \supseteq P_m$ outside $(I_{m+1},J_{m+1})$ such that for any $\mathcal{F}$-term $t(\overline{x},y)$ involving $P_m$, $t(\overline{a_k},b)$ is decomposed on $I_{m+1} + (c) + J_{m+1}$ using some $\mathcal{F}$-terms $\overline{r}(\overline{x})$ and $\overline{s}(y)$ involving $P_{m+1}$. (Note that we can always ensure nonempty $I_{m+1}$ and $J_{m+1}$ due to \autoref{rmk:SaturatedIndices}.)
    We define $I' = \bigcap_{i\in\omega} I_i$, $J' = \bigcap_{i\in\omega} J_i$, and $P = \bigcup_{i\in\omega} P_i$, which satisfy the required conditions.
\end{proof}

\begin{thm}\label{thm:DistalityCriterion}
    $\widetilde{\mathcal{N}}$ is distal if the following conditions hold:
    \begin{enumerate}
        \item $\mathcal{N}$ and $\mathcal{R}$ are distal.
        \item For any $M \subseteq N$ with $\langle M \rangle_{\mathcal{F}} = M$ and any $b,b' \in R$, if $\mathrm{tp}^{\mathcal{R}}(b / M \cap R) = \mathrm{tp}^{\mathcal{R}}(b' / M \cap R)$, then $\mathrm{tp}^{\widetilde{\mathcal{N}}}(b / M) = \mathrm{tp}^{\widetilde{\mathcal{N}}}(b' / M)$.
        \item For any $M \subseteq N$ with $\langle M \rangle_{\mathcal{F}} = M$ and any $b,b' \in N$, if $\mathrm{tp}^\mathcal{N}(b / M) = \mathrm{tp}^\mathcal{N}(b' / M)$ and $\lambda(x) \in M$ for any $x \in \langle b M \rangle_{\mathcal{F}_N}$, then $\mathrm{tp}^{\widetilde{\mathcal{N}}}(b / M) = \mathrm{tp}^{\widetilde{\mathcal{N}}}(b' / M)$.
        \item $\widetilde{\mathcal{N}}$ admits variable separation.
    \end{enumerate}
\end{thm}

\begin{proof}
    Take a sequence $(\overline{a_k})_{k\in I+(c)+J}$ and a singleton $b$ satisfying the assumptions in \autoref{dfn:Distality}.

    \setcounter{clm}{0}
    \begin{clm}
        If $b \in R$, then $(\overline{a_k})_{k\in I+(c)+J}$ is indiscernible over $b$.
    \end{clm}

    \begin{prooc}
        By \autoref{lem:BasicsofDistality} (4), it suffices to show that $\mathrm{tp}^{\widetilde{\mathcal{N}}}(b,\overline{a_k}) = \mathrm{tp}^{\widetilde{\mathcal{N}}}(b,\overline{a_c})$ for any $k \in I+J$.
        Let $M_k = \langle \overline{a_k} \rangle_{\mathcal{F}}$.
        Since $\mathcal{R}$ is $\emptyset$-definable in $\widetilde{\mathcal{N}}$:
            \begin{itemize}
            \item $(M_k \cap R)_{k\in I+(c)+J}$ is indiscernible in $\mathcal{R}$.
            \item $(M_k \cap R)_{k\in I+J}$ is indiscernible over $b$ in $\mathcal{R}$.
        \end{itemize}
        It follows from the distality of $\mathcal{R}$ that $(M_k \cap R)_{k\in I+(c)+J}$ is indiscernible over $b$ in $\mathcal{R}$.
        For arbitrary $k \in I+J$, we extend the partial elementary map $f:M_k \to M_c$ to an automorphism of $\widetilde{\mathcal{N}}$ and let $b' = f^{-1}(b)$.
        Since $\mathrm{tp}^{\mathcal{R}}(b,M_k \cap R) = \mathrm{tp}^{\mathcal{R}}(b,M_c \cap R)$, we have $\mathrm{tp}^{\mathcal{R}}(b / M_k \cap R) = \mathrm{tp}^{\mathcal{R}}(b' / M_k \cap R)$.
        Condition (2) then implies $\mathrm{tp}^{\widetilde{\mathcal{N}}}(b / M_k) = \mathrm{tp}^{\widetilde{\mathcal{N}}}(b' / M_k)$, which is equivalent to $\mathrm{tp}^{\widetilde{\mathcal{N}}}(b,M_k) = \mathrm{tp}^{\widetilde{\mathcal{N}}}(b,M_c)$.
    \end{prooc}

    \begin{clm}
        For any $b \in N$, $(\overline{a_k})_{k\in I+(c)+J}$ is indiscernible over $b$.
    \end{clm}

    \begin{prooc}
        We choose a truncation $(I',J') \subset_{\mathrm{t}} (I,J)$ and a parameter set $P$ outside $(I',J')$ that satisfy the consequence of \autoref{lem:StrongerSeparationofTerms}.
        For each $k \in I' + (c) + J'$, define $A_k = \langle \overline{a_k} P \rangle_{\mathcal{F}}$.
        Also, let $B = \langle b P \rangle_{\mathcal{F}} \cap R$.
        These definitions ensure the following:
        \begin{itemize}
            \item $(A_k)_{k\in I'+(c)+J'}$ is indiscernible.
            \item $(A_k)_{k\in I'+J'}$ is indiscernible over $Bb$.
            \item $B \subseteq R$.
        \end{itemize}
        By combining these conditions, Claim 1, and the proof of \autoref{lem:BasicsofDistality} (1) and (3), it follows that $(A_k)_{k\in I'+(c)+J'}$ is indiscernible over $B$.
        The distality of $\mathcal{N}$ implies that $(A_k)_{k\in I'+(c)+J'}$ is indiscernible over $Bb$ in $\mathcal{N}$.

        Now, let $M_k = \langle A_k B \rangle_{\mathcal{F}_N}$.
        For any $k \in I'+J'$, note that the map $f:M_k \to M_c$ is partial elementary in $\widetilde{\mathcal{N}}$ because of the indiscernibility of $(A_k)_{k\in I'+(c)+J'}$ over $B$.
        We also have $\mathrm{tp}^\mathcal{N}(b,M_k) = \mathrm{tp}^\mathcal{N}(b,M_c)$.
        Furthermore, for any $\overline{x} \in \langle b M_k \rangle_{\mathcal{F}_N}$ and any $h \in \mathcal{F}_R \cup \left\{ \lambda \right\}$, $h(\overline{x})$ is represented as $u(\overline{r}(\overline{a_k}),\overline{s}(b))$ for some $\mathcal{F}_N$-term $u$ and $\mathcal{F}$-terms $\overline{r}$ and $\overline{s}$ involving $P$, where $\overline{s}(b) \in R$.
        We then have $\overline{r}(\overline{a_k}) \in A_k$ and $\overline{s}(b) \in B$, which imply $h(\overline{x}) \in M_k$.
        In particular, we have $\langle M_k \rangle_{\mathcal{F}} = M_k$ and $\lambda(x) \in M_k$ for all $x \in \langle b M_k \rangle_{\mathcal{F}_N}$.
        Therefore, applying condition (3), we conclude that $\mathrm{tp}^{\widetilde{\mathcal{N}}}(b,M_k) = \mathrm{tp}^{\widetilde{\mathcal{N}}}(b,M_c)$.
    \end{prooc}
\end{proof}

\section{\texorpdfstring{$(\mathbb{Z};<,+)$}{(Z;<,+)} with an almost sparse sequence}
\label{chp:ZR}

This chapter is devoted to the study of the structure $(\mathbb{Z};<,+,R)$, consisting of the ordered additive group of integers expanded by an `almost sparse' sequence $R$.
This structure has been analyzed extensively by several authors, including Lambotte, Point, Sem\"enov, and Tong \cite{Semenov1980,Point2000,Lambotte2020,Lambotte2021,Tong2025}.

Following Sem\"enov's foundational work, Point \cite{Point2000} established quantifier elimination for this structure under the additional assumption that $R$ is `congruence-periodic.'
Subsequently, Lambotte and Point\cite{Lambotte2020} proved the NIP for this structure under the same assumption, and Lambotte \cite{Lambotte2021} later extended this result to the general case without the congruence periodicity assumption.
More recently, Tong \cite{Tong2025} demonstrated the distality of $(\mathbb{Z};<,+,R)$ assuming congruence periodicity.

In this chapter, we establishes several results that allow us to apply \autoref{thm:DistalityCriterion} to prove the distality of $(\mathbb{Z};<,+,R)$ without the assumption of congruence periodicity, thereby strengthening previous findings and answering a question posed by Tong (ibid., Question 1.4).

The chapter is organized as follows.
After examining the basic properties of almost sparse sequences in \autoref{sec:ZRIntro}, we conduct a back-and-forth argument in \autoref{sec:ZRBaF} to establish two propositions corresponding to conditions (2) and (3) of \autoref{thm:DistalityCriterion}.
The approach in \autoref{sec:ZRBaF} is inspired by Point's work on quantifier elimination \cite[\S3]{Point2000}.
Our propositions may be viewed as a weaker form of her result that does not require the sequence to be congruence-periodic.
\autoref{sec:ZRTerm} is devoted to variable separation.
While a similar argument appeared in the work by Lambotte and Point \cite[Claim 2.33]{Lambotte2020}, we provide a version specifically suited for proving distality.
Finally, in \autoref{sec:ZRDistal}, we conclude that $(\mathbb{Z};<,+,R)$ is distal whenever $R$ is almost sparse.

Our analysis of $(\mathbb{Z};<,+,R)$ is also significant for the expansion of the $p$-adic field $(\mathbb{Q}_p;+,\cdot,p^\mathbb{Z},p^R)$, which we discuss later.
Indeed, $(\mathbb{Q}_p;+,\cdot,p^\mathbb{Z},p^R)$ interprets the structure $\left(p^\mathbb{Z};<_v,\cdot,p^R\right)$, which is isomorphic to $(\mathbb{Z};<,+,R)$.

\subsection{Basic properties of almost sparse sequences}
\label{sec:ZRIntro}

We investigate the key properties of almost sparse sequences necessary for our distality arguments.
Let $R = (r_n)_{n\in \mathbb{N}}$ be a strictly increasing sequence of integers.
We define the successor function $S(r_n) = r_{n+1}$ and its inverse $S^{-1}(r_{n+1}) = r_n$ (with $S^{-1}(r_0) = r_0)$.

\begin{dfn}\label{dfn:Operators}
    An \emph{operator} on $R$ is a term of the form $A(x) = \sum_{i=-n}^m z_i S^i(x)$, where $z_i \in \mathbb{Z}$.
    For operators $A(x)$ and $B(x)$, we write $A >_{ae} B$ if $A(a) > B(a)$ for all but finitely many $a \in R$.
    The relations $A <_{ae} B$ and $A =_{ae} B$ are defined similarly.
\end{dfn}

\begin{dfn}\label{dfn:AlmostSparseSequence}
    $R = (r_n)_{n\in\mathbb{N}}$ is \emph{almost sparse} if the following conditions hold:
    \begin{enumerate}
        \item For any operators $A(x)$ and $B(x)$, one of $A >_{ae} B$, $A <_{ae} B$, or $A =_{ae} B$ holds.
        \item For any operators $A(x)$ and $B(x)$ such that $A >_{ae} B$, there exists an integer $\Delta \geq 0$ such that $A(S^\Delta(x)) >_{ae} B(S^\Delta(x)) +  x$.
    \end{enumerate}
\end{dfn}

\begin{rmk}\label{rmk:BasicsofAlmostSparseSequence}
    It is clear from the definition that $A(x) >_{ae} B(x) \Leftrightarrow A(S^z(x)) >_{ae} B(S^z(x))$ for any operators $A(x)$ and $B(x)$ and any $z \in \mathbb{Z}$, with similar equivalence for $<_{ae}$ and $=_{ae}$.
\end{rmk}

\begin{ex}
    The sequences $(2^n)_{n\in \mathbb{N}}$ and $(n!)_{n\in \mathbb{N}}$ are easily seen to be almost sparse.
    Other examples include the Fibonacci sequence.
    In constrast, $(n)_{n\in\mathbb{N}}$ and $(2^n + n)_{n\in\mathbb{N}}$ are not almost sparse.
    For the details, see \cite[\textsection 2]{Point2000}.
\end{ex}

Almost sparseness can also be characterized in another way.
Point and Lambotte \cite{Lambotte2020} introduced the notion of regular sequence.
$(r_n)_{n\in \mathbb{N}}$ is \emph{almost regular} if $\lim_{n\to \infty} \frac{r_{n+1}}{r_n} = \theta$ for some real number $\theta > 1$ or $\theta = \infty$, and additionally if $\theta$ is an algebraic number with minimal polynomial $\sum_{i=0}^d a_i x^i$ over $\mathbb{Q}$, then $\sum_{i=0}^d a_i S^i(x) =_{ae} 0$.
It has been shown in \cite[Lemma 2.26]{Lambotte2020} and \cite[Theorem 2.27]{Tong2025} that almost sparseness and almost regularity are equivalent.

We henceforth assume that $R$ is almost sparse in the rest of this paper.

\begin{lem}\label{lem:OperatorBounded}
    If an operator $A(x)$ satisfies $A >_{ae} 0$, then there exists $m \in \mathbb{Z}$ such that
    either $A =_{ae} S^m(x)$ or $S^m(x) <_{ae} A <_{ae} < S^{m+1}(x)$.
\end{lem}

\begin{proof}
    \setcounter{clm}{0}
    \begin{clm}
        For each $n \geq 1$, there exists $k \geq 1$ such that $nx <_{ae} S^k(x)$.
    \end{clm}

    \begin{prooc}
        Since $S(x) >_{ae} x$, by almost sparseness, there exists $\Delta \geq 0$ such that $S^{\Delta+1}(x) >_{ae} S^\Delta(x) + x$.
        Thus, $2x \leq x + S^\Delta(x) <_{ae} S^{\Delta+1}(x)$.
        By iteration, it follows that $4x <_{ae} 2S^{\Delta+1}(x) <_{ae} S^{2\Delta+2}(x)$, $8x <_{ae} 2S^{2\Delta+2}(x) <_{ae} S^{3\Delta+3}(x)$, and so on, proving the claim.
    \end{prooc}
    
    Now, we consider an operator $A(x) = \sum_{i=-n}^m z_i S^i(x)$.
    By Claim 1, there exists $k \geq 1$ such that $z_i S^i(x) <_{ae} S^k(x)$ for all $-n \leq i \leq m$.
    We can also find $l \geq 1$ such that $(m+n+1)S^k(x) <_{ae} S^l(x)$.
    These imply $A <_{ae} S^l(x)$.
    Conversely, since $A >_{ae} 0$, there exists $\Delta \geq 0$ such that $A(S^\Delta(x)) >_{ae} x$, or equivalently $A >_{ae} S^{-\Delta}(x)$.
    Having established $S^{-\Delta}(x) <_{ae} A <_{ae} S^l(x)$, the existence of the desired $m$ follows by taking the maximal integer such that $S^m(x) \leq_{ae} A$.
\end{proof}

We consider the structure $(\mathbb{Z};<,+,R)$, where $R$ is regarded as a subset $\left\{ r_n \mid n \in \mathbb{N} \right\}$ of $\mathbb{Z}$.

\begin{dfn}\label{dfn:LambdaRSS-1}
    We define the functions $\lambda$, $S$, and $S^{-1}$ on $\mathbb{Z}$ on $\mathbb{Z}$ as follows:
    \begin{itemize}
        \item $\lambda(x) = \mathrm{max}\left\{ r \in R \mid r \leq x \right\}$ if $x \geq r_0$ and $\lambda(x) = r_0$ otherwise.
        \item $S(x) = S(\lambda(x))$ and $S^{-1}(x) = S^{-1}(\lambda(x))$.
    \end{itemize}
\end{dfn}

Note that $\mathrm{Im}(\lambda) = R$, and that these functions are $\emptyset$-definable in $(\mathbb{Z};<,+,R)$.
For convenience, We regard $S^0(x)$ as $\lambda(x)$, so that $A(x) = A(\lambda(x))$ for any operator $A$ and $x \in \mathbb{Z}$.

Let $\mathcal{U}$ be a monster model of $(\mathbb{Z};<,+,R)$.
Following usual notation, we denote by $R^\mathcal{U}$ the interpretation of the symbol $R$ in $\mathcal{U}$.
We refer to $\mathbb{Z} \subseteq U$ as the standard part of $\mathcal{U}$.
If $x > 0$ is not in standard part, then $A >_{ae} B \iff A(x) > B(x)$ for any operators $A$ and $B$, with similar properties holding for $<_{ae}$ and $=_{ae}$.

\begin{dfn}
    For positive elements $a, b \in U$, we define the following relations:
    \begin{itemize}
        \item $a \approx b$ if either ($a,b \in \mathbb{Z}$), or ($a,b \notin \mathbb{Z}$ and there exists a natural number $m$ such that $S^{-m}(a) < b < S^m(a)$).
        \item $a \ll b$ if $a < b$ and $a \not\approx b$, which is equivalent to that for all $m \in \mathbb{N}$, $S^m(a) < S^{-m}(b)$.
    \end{itemize}
\end{dfn}

It is clear that $\approx$ is an equivalence relation.

\begin{lem}\label{lem:ApproxAndLL}
    \begin{enumerate}
        \item For any positive element $a \in U \setminus \mathbb{Z}$ and any operator $A >_{ae} 0$, $a \approx A(a)$.
        \item For any positive elements $a,b \in U$ such that $a \ll b$, and for any operators $A <_{ae} C <_{ae} B$, we have $A(b) < C(b) \pm a < B(b)$. In particular, $b \pm a \approx b$.
    \end{enumerate}
\end{lem}

\begin{proof}
    The first claim follows directly from \autoref{lem:OperatorBounded}.
    For the second, note that $b \notin \mathbb{Z}$.
    Since $B - C >_{ae} 0$, by almost sparseness, there exists $\Delta \geq 0$ such that $B - C >_{ae} S^{-\Delta}(x)$.
    It follows that $B(b) - C(b) > S^{-\Delta}(b) > a$.
    A similar argument shows $C(b) - A(b) > a$.
\end{proof}

\begin{lem}\label{lem:ElementInROutsideOfM}
    Let $M \subseteq U$ be a subgroup containing $\mathbb{Z}$ that is closed under $\lambda$, $S$, and $S^{-1}$.
    For $a \in M$, $b \in R^\mathcal{U} \setminus M$, and an operator $A$, the following hold:
    \begin{enumerate}
        \item Either $|a| \ll b$ or $b \ll |a|$.
        \item Either $|A(b) + a| \approx b$ or $|A(b) + a| \approx |a|$.
        \item $A(b) + a > 0$ if and only if ($|a| \ll b$ and $A >_{ae} 0$), ($b \ll a$), or ($A =_{ae} 0$ and $a > 0$).
        \item The set $\widetilde{M} = \left\{ A(b) + a \mid \text{$A$ is an operator, } a \in M \right\}$ is a subgroup closed under $\lambda$, $S$, and $S^{-1}$.
        \item If $A(b) + a \in R^\mathcal{U}$, then it is either in $M \cap R^\mathcal{U}$ or equals $S^z(b)$ for some $z \in \mathbb{Z}$.
    \end{enumerate}
\end{lem}

\begin{proof}
    If $|a| \approx b$, then $b$ would be $S^z(|a|) \in M$ for some $z \in \mathbb{Z}$, which is a contradiction; thus (1) holds.
    Claim (2) follows from \autoref{lem:ApproxAndLL} if $A \neq_{ae} 0$, and is trivial otherwise.
    Claim (3) is a direct consequence of (1) and \autoref{lem:ApproxAndLL}.
    For (4), $\widetilde{M}$ is clearly a subgroup.
    If $A(b) + a \leq 0$, then $\lambda(A(b) + a) \in \mathbb{Z} \subseteq M$.
    Otherwise, by applying (2), one can see that $\lambda(A(b) + a)$ is either in $M$ or equals $S^z(b)$ for some $z \in \mathbb{Z}$.
    These imply (4) for $\lambda$ and (5).
    Claim (4) for $S$ and $S^{-1}$ follow since $S = S \circ \lambda$ and $S^{-1} = S^{-1} \circ \lambda$.
\end{proof}
\subsection{Back-and-forth}
\label{sec:ZRBaF}

Quantifier elimination for the ordered additive group of integers $(\mathbb{Z};<,+)$ is a classical result.

\begin{thm}\label{thm:QEofPA}
    The structure $(\mathbb{Z};<,+)$ admits quantifier elimination in the language $\left\{ 0,1,+,-,<,(\equiv_n)_{n\geq 2} \right\}$, where $\equiv_n$ denotes the congruence relation modulo $n$.
\end{thm}

A proof can be found in, for example, Marker's textbook \cite[Section 3.1]{Marker2002}.
We utilize this result in the proof of \autoref{lem:ExtendabilityforR}.
Our goal is to establish a back-and-forth system for our structure $(\mathbb{Z};<,+,R)$ to identify partial elementary embeddings.

\begin{dfn}\label{dfn:BackandForthSystem}
    Let $\mathcal{M}$ and $\mathcal{N}$ be $L$-structures.
    A \emph{back-and-forth system} between $\mathcal{M}$ and $\mathcal{N}$ is a nonempty set $\mathcal{I}$ of partial $L$-embeddings from $\mathcal{M}$ to $\mathcal{N}$ satisfying the following conditions:
    \begin{enumerate}
        \item Forth: For any $a \in M$ and $f \in \mathcal{I}$, there exists $g \in \mathcal{I}$ that expands $f$ and includes $a$ in its domain.
        \item Back: For any $b \in N$ and $f \in \mathcal{I}$, there exists $g \in \mathcal{I}$ that expands $f$ and includes $b$ in its image.
    \end{enumerate}
\end{dfn}

\begin{thm}\label{thm:BackandForthSystem}
    If $\mathcal{I}$ is a back-and-forth system, then every $f \in \mathcal{I}$ is a partial $L$-elementary embedding.
\end{thm}

The proof of \autoref{thm:BackandForthSystem} proceeds by a straightforward induction on the complexity of $L$-formulas, and we omit it.
Throughtout this paper, we consider back-and-forth systems from a monster model to itself, assuming that every map in the system has small domain (see \autoref{rmk:Monster}).

\begin{ntt}
    For $n \geq 2$ and $0 \leq i < n$, let $\cdot \equiv_n i$ denote the unary relation on $R$ identifying elements $x \in R$ such that $x \equiv i \pmod n$.
\end{ntt}

In our monster model $\mathcal{U}$ of $(\mathbb{Z};<,+,R)$, $(R^\mathcal{U};<,(\cdot \equiv_n i)_{n,i})$---consisting of $R^\mathcal{U}$ with the induced order and residue predicates---plays a pivotal role in our back-and-forth system.

\begin{dfn}\label{dfn:BackandForthSystemforZ}
    Let $L_R = \left\{ 0,1,<,+,-,R,\lambda,S,S^{-1} \right\}$.
    We define $\mathcal{I}$ as the set of partial functions $f:\mathcal{U} \to \mathcal{U}$ satisfying the following:
    \begin{enumerate}
        \item $\mathrm{Dom}(f)$ is an $L_R$-substructure.
        \item $f$ is a partial $L_R$-embedding.
        \item $f$ is a partial elementary embedding of $(U;<,+)$.
        \item $f|_{\mathrm{Dom}(f)\cap R^\mathcal{U}}$ is a partial elementary embedding of
        $(R^\mathcal{U};<,(\cdot \equiv_n i)_{n,i})$.
    \end{enumerate}

    Additionally, we consider the condition:

    \begin{enumerate}
        \item[2'.] $a \in R^{\mathcal{U}} \Rightarrow f(a) \in R^{\mathcal{U}}$ for any $a \in \mathrm{Dom}(f)$.
    \end{enumerate}
\end{dfn}

\begin{lem}\label{lem:ConditionstobeinI}
    Suppose a partial function $f:\mathcal{U} \to \mathcal{U}$ satisfies conditions (1), (2'), (3), and (4) of\autoref{dfn:BackandForthSystemforZ}.
    Then $f$ is a partial $L_R$ embedding and thus $f \in \mathcal{I}$.
\end{lem}

\begin{proof}
    We must show that $f(a) \in R^{\mathcal{U}} \Rightarrow a \in R^{\mathcal{U}}$ for any $a \in \mathrm{Dom}(f)$, and that $f$ commutes with $\lambda$, $S$, and $S^{-1}$.
    Take $a \in \mathrm{Dom}(f)$ with $a \geq r_0$.
    Then $\lambda(a)$, $S(a) \in \mathrm{Dom}(f)$ by (1), and $f(\lambda(a)) \leq f(a) < f(S(a))$ by (3).
    In the structure $R^\mathcal{U}$, $S(a)$ is the successor of $\lambda(a)$, and this fact can be represented by a sentence.
    By (4), $f(S(a))$ must be the successor of $f(\lambda(a))$ in $R^\mathcal{U}$, which implies that $\lambda(f(a)) = f(\lambda(a))$ and that $S(f(a)) = f(S(a))$.
    A similar argument applies to $S^{-1}$, showing $S^{-1}(f(a)) = f(S^{-1}(a))$.
    Finally, if $f(a) \in R^\mathcal{U}$, then $f(a) = \lambda(f(a)) = f(\lambda(a))$, which implies $a = \lambda(a) \in R^\mathcal{U}$.
\end{proof}

Our objective is to prove that $\mathcal{I}$ forms a back-and-forth system in the language $\left\{ <,+,R \right\}$.
Since $f^{-1} \in \mathcal{I}$ whenever $f \in \mathcal{I}$, it suffices to prove the `forth' direction.
We first consider the case where the added element belongs to $R^\mathcal{U}$.

\begin{lem}\label{lem:ExtendabilityforR}
    Let $f \in \mathcal{I}$ with $\mathrm{Dom}(f)$ denoted by $M$, and take $b \in R^\mathcal{U} \setminus M$.
    Then there exists $c \in R^\mathcal{U}$ such that:
    \begin{equation*}
        \mathrm{tp}^{(R^\mathcal{U};<,(\cdot \equiv_n i))}(b, M \cap R^\mathcal{U}) = \mathrm{tp}^{(R^\mathcal{U};<,(\cdot \equiv_n i))}(c, f(M \cap R^\mathcal{U})).
    \end{equation*}
    Moreover, for any such $c$, there exists an extension $\widetilde{f} \in \mathcal{I}$ of $f$ such that $\widetilde{f}(b) = c$.
\end{lem}

\begin{proof}
    Since $\mathcal{U}$ is a monster model, the structure $(R^\mathcal{U};<,(\cdot \equiv_n i)_{n,i})$ is sufficiently saturated.
    This saturation, combined with \autoref{dfn:BackandForthSystemforZ} (4), allows us to find such an element $c \in R^\mathcal{U}$.
    We now show that $f$ can be extended in a desired way.

    \setcounter{clm}{0}
    \begin{clm}
        For any $a \in M$, $a \lessgtr b \iff f(a) \lessgtr c$.
    \end{clm}

    \begin{prooc}
        We may assume $a > \mathbb{Z}$ since $b,c > \mathbb{Z}$.
        Then $\lambda(a) \leq a < S(a)$.
        Recall from \autoref{lem:ElementInROutsideOfM} (1) that either $a \ll b$ or $b \ll a$.
        If $b \ll a$, then $b < \lambda(a)$, and by our choice of $c$, we have $c < f(\lambda(a)) \leq f(a)$.
        Conversely, if $a \ll b$, then $S(a) < b$, which implies $f(a) < f(S(a)) < c$.
    \end{prooc}

    For any operator $A(x)$ and any $a \in M$, Claim 1 and \autoref{lem:ElementInROutsideOfM} (3) ensure that $A(b) + a \gtreqqless 0 \Leftrightarrow A(c) + f(a) \gtreqqless 0$.
    Furthermore, by the choice of $c$, we have $S^z(b) \equiv_n S^z(c)$ for all $z \in \mathbb{Z}$ and $n \geq 2$, while $a \equiv_n f(a)$ holds by \autoref{dfn:BackandForthSystemforZ} (3).
    Thus, $A(b) + a \equiv_n  A(c) + f(a)$ for all $n \geq 2$.

    We define $\widetilde{M} = \left\{ A(b) + a \mid \text{$A$ is an operator and } a \in M \right\}$ and extend $f$ to $\widetilde{f}:\widetilde{M} \to \mathcal{U}$ by setting $\widetilde{f}(A(b) + a) = A(c) + f(a)$.
    The observations above ensure that $\widetilde{f}$ is a well-defined $\left\{ 0,1,+,-,<,(\equiv_n)_{n\geq 2} \right\}$-embedding.

    \begin{clm}
        $\widetilde{f}$ is in $\mathcal{I}$.
    \end{clm}

    \begin{prooc}
        \autoref{dfn:BackandForthSystemforZ} (1) is satisfied by \autoref{lem:ElementInROutsideOfM} (4).
        Since $\widetilde{f}$ is a partial elementary embedding of $(U;<,+)$ by \autoref{thm:QEofPA}, \autoref{dfn:BackandForthSystemforZ} (3) holds.
        By \autoref{lem:ElementInROutsideOfM} (5), $\widetilde{M} \cap R^\mathcal{U} = (M \cap R^\mathcal{U}) \cup \left\{ S^z(b) \mid z \in \mathbb{Z} \right\}$, satisfying \autoref{dfn:BackandForthSystemforZ} (2').
        Since this intersection is contained in the definable closure of $\left\{ b \right\} \cup (M \cap R^\mathcal{U})$ within $(R^\mathcal{U};<,(\cdot \equiv_n i)_{n,i})$, \autoref{dfn:BackandForthSystemforZ} (4) follows from the elementarity of $\widetilde{f}|_{\left\{ b \right\} \cup (M \cap R^\mathcal{U})}$ in this structure, which is equivalent to the assumption on types of $b$ and $c$.
    \end{prooc}
\end{proof}

\begin{rmk}\label{rmk:DomainIsTerms}
    Note that in \autoref{lem:ExtendabilityforR}, we constructed $\widetilde{f} \in \mathcal{I}$ such that $\mathrm{Dom}(\widetilde{f}) = \left\{ A(b) + a \mid \text{$A$ is an operator and } a \in \mathrm{Dom}(f) \right\}$.
\end{rmk}

Next, we show that any element can be added to the domain.

\begin{lem}\label{lem:SubLemmaforExtendabilityforAll}
    For any $f \in \mathcal{I}$ and any $b \in \mathcal{U}$, there exists an extension $\widetilde{f} \in \mathcal{I}$ of $f$ such that for any $x \in \langle b, \mathrm{Dom}(\widetilde{f}) \rangle_{\left\{ +,- \right\}}$, we have $\lambda(x) \in \mathrm{Dom}(\widetilde{f})$.
\end{lem}

\begin{proof}
    Let $T = \langle b, \mathrm{Dom}(f) \rangle_{\left\{ +,-,\lambda,S,S^{-1} \right\}}$.
    By applying \autoref{lem:ExtendabilityforR} transfinitely many times, we can expand $f$ to $\widetilde{f} \in \mathcal{I}$ such that $\mathrm{Dom}(\widetilde{f})$ contains $T \cap R^\mathcal{U}$.
    From \autoref{rmk:DomainIsTerms}, we have $\mathrm{Dom}(\widetilde{f}) \subseteq T$.
    It follows that $\langle b, \mathrm{Dom}(\widetilde{f}) \rangle_{\left\{ +,- \right\}} \subseteq T$, and thus for any $x \in \langle b, \mathrm{Dom}(\widetilde{f}) \rangle_{\left\{ +,- \right\}}$, we have $\lambda(x) \in T \cap R^\mathcal{U} \subseteq \mathrm{Dom}(\widetilde{f})$.
\end{proof}

\begin{lem}\label{lem:ExtendabilityforAll}
    Suppose $b \in \mathcal{U}$ and $f \in \mathcal{I}$ satisfy that for any $x \in \langle b, \mathrm{Dom}(f) \rangle_{\left\{ +,- \right\}}$, we have $\lambda(x) \in \mathrm{Dom}(f)$.
    Then there exists $c \in \mathcal{U}$ such that:
    \begin{equation*}
        \mathrm{tp}^{\left( U;<,+ \right)} (b, \mathrm{Dom}(f)) = \mathrm{tp}^{\left( U;<,+ \right)} (c, \mathrm{Im}(f)).
    \end{equation*}
    Moreover, for any such $c$, there exists an extension $\widetilde{f} \in \mathcal{I}$ of $f$ such that $\widetilde{f}(b) = c$.
\end{lem}

\begin{proof}
    By \autoref{dfn:BackandForthSystemforZ} (3) and the saturation of $\mathcal{U}$, such a $c \in \mathcal{U}$ exists.
    We extend $f$ to a partial $\left\{ <,+ \right\}$-elementary embedding $\widetilde{f}:\langle b, \mathrm{Dom}(f) \rangle_{\left\{ +,- \right\}} \to \mathcal{U}$ in the obvious way.
    We prove $\widetilde{f} \in \mathcal{I}$.
    \autoref{dfn:BackandForthSystemforZ} (3) is trivial.
    By our assumption, $\langle b, \mathrm{Dom}(f) \rangle_{\left\{ +,- \right\}}$ is closed under $\lambda,S,S^{-1}$, so \autoref{dfn:BackandForthSystemforZ} (1) holds.
    Furthermore, $\langle b, \mathrm{Dom}(f) \rangle_{\left\{ +,- \right\}} \cap R^\mathcal{U}$ = $\mathrm{Dom}(f) \cap R^\mathcal{U}$, which ensures \autoref{dfn:BackandForthSystemforZ} (2') and (4).
\end{proof}

Therefore, we have:

\begin{thm}\label{thm:PEEinR}
    $\mathcal{I}$ forms a back-and-forth system in the language $\left\{ <,+,R \right\}$.
    Consequently, every $f \in \mathcal{I}$ is a partial $\left\{ <,+,R \right\}$-elementary embedding.
\end{thm}
\subsection{Variable separation}
\label{sec:ZRTerm}

Let $(\overline{a_k})_{k\in I+(c)+J}$ be a sequence and $b$ a singleton satisfying the assumptions in \autoref{dfn:Distality}.
As noted in \autoref{rmk:SaturatedIndices}, we assume $I$ and $J$ are sufficiently saturated dense linear orders without endpoints.
We keep following the terminology in \autoref{dfn:IndisParameter}.
In this section, we refer to $\left\{ +,-,\lambda,S,S^{-1} \right\}$-terms as terms for brevity.

\begin{lem}\label{lem:TermSeparationforZR}
    Let $t(\overline{x},y)$ be a term.
    Then there exist a truncation $(I',J') \subset_{\mathrm{t}} (I,J)$ and terms $u(\overline{x})$ and $r(y)$ involving parameters outside $(I',J')$, such that:
    \begin{equation*}
        t(\overline{a_k},b) = u(\overline{a_k}) + r(b)
    \end{equation*}
    for all $k \in I'+(c)+J'$.
    In particular, if $t$ is of the form $\lambda(t')$, $S(t')$, or $S^{-1}(t')$, then either $u$ or $r$ is $0$.
\end{lem}

\begin{proof}
    First, note that for any terms $u(\overline{x})$ and $r(y)$, the sign of $u(\overline{a_k}) + r(b)$ remains constant as $k$ ranges over $I + (c) + J$.
    Indeed, the assumptions on indiscernibility imply the monotonicity of $(u(\overline{a_k}))_{k\in I+(c)+J}$ and the constancy of the sign of $(u(\overline{a_k}) + r(b))_{k\in I+J}$.
    We will make repeated use of this fact implicitly.

    We proceed by induction on the complexity of $t$.
    The nontrivial case is when $t$ is the form of $\lambda(t')$ (the cases for $S$ and $S^{-1}$ follow since $S = S \circ \lambda$ and $S^{-1} = S^{-1} \circ \lambda$).
    By the inductive hypothesis, we may write $t'(\overline{a_k},b) = u(\overline{a_k}) + r(b)$.
    We may assume $u(\overline{a_k}) + r(b) > 0$ for all $k \in I+(c)+J$, since otherwise $\lambda(u(\overline{a_k}) + r(b))$ constantly equals an integer.

    \setcounter{case}{0}
    \begin{case}
        $(u(\overline{a_k}))_{k\in I+(c)+J}$ is constant.

        By choosing $i_0 \in I$ and shrinking $I$ to a final segment $I'$ such that $i_0 \notin I'$, we can denote $u(\overline{a_k}) + r(b)$ as $r'(b)$, where $r'(y) = u(\overline{a_{i_0}}) + r(y)$ is a term involving parameters outside $(I',J)$.
    \end{case}

    \begin{case}
        $(u(\overline{a_k}))_{k\in I+(c)+J}$ is strictly monotonic.

        Suppose this sequence is strictly increasing (the decreasing case is symmetric by shrinking $J$).
        We choose some $i_0 \in I$ and shrink $I$ to $I'$ such that $i_0 \notin I'$. Then:
        \begin{equation*}
            u(\overline{a_k}) + r(b) = (u(\overline{a_k}) - u(\overline{a_{i_0}})) + (u(\overline{a_{i_0}}) + r(b)).
        \end{equation*}
        It allows us to redefine $u$ and $r$ so that $u(\overline{a_k})$ and $r(b)$ are both positive.

        If $u(\overline{a_k}) > S^m(r(b))$ for any $k$ and $m$, then by \autoref{lem:ApproxAndLL} (2), we have $\lambda(u(\overline{a_k})) < u(\overline{a_k}) + r(b) < S^2(u(\overline{a_k}))$.
        Since the sign of $S(u(\overline{a_k})) - (u(\overline{a_k}) + r(b))$ is constant, $\lambda(u(\overline{a_k}) + r(b))$ equals either $\lambda(u(\overline{a_k}))$ of $S(u(\overline{a_k}))$ independently of $k$.

        If there exists $m$ such that for all $k$, we have $u(\overline{a_k}) \leq S^m(r(b))$, then $u(\overline{a_k}) + r(b) \approx r(b)$ by \autoref{lem:ApproxAndLL} (1).
        It implies that there exists $z \in \mathbb{Z}$ such that for all $k$, we have $S^z(r(b)) \leq u(\overline{a_k}) + r(b) < S^{z+1}(r(b))$.
        Thus $\lambda(u(\overline{a_k}) + r(b)) = S^z(r(b))$.
    \end{case}
\end{proof}
\subsection{Distality}
\label{sec:ZRDistal}

We established a back-and-forth system for $(\mathbb{Z};<,+,R)$ in \autoref{sec:ZRBaF}.
Intuitively, it demonstrates that the model-theoretic behavior of the expansion is determined by its components: $(\mathbb{Z};<,+)$ and $(R;<,(\cdot \equiv_n i)_{n,i})$.
It is known that these structures are dp-minimal, and therefore distal.

\begin{thm}[{Aschenbrenner, Dolich, Haskell, Macpherson, and Starchenko \cite{Aschenbrenner2016}}]\label{thm:NIPofPA}
    $(\mathbb{Z};<,+)$ is dp-minimal.
\end{thm}

\begin{thm}[{Simon \cite[a special case of Proposition 4.2]{Simon2011}}]\label{thm:ColoredOrderDpminimality}
    Let $M$ be a linear order and $(C_i)_{i<\alpha}$ be a family of unary subsets of $M$.
    Then, $(M;<,(C_i)_{i<\alpha})$ is dp-minimal.
\end{thm}

A concise proof of \autoref{thm:ColoredOrderDpminimality} can be found in Simon's textbook \cite[Proposition A.2]{Simon2015}.
The dp-minimality of these structures implies their distality (see \autoref{rmk:DistalImplications}).

\begin{thm}
    $(\mathbb{Z};<,+,R)$ is distal.
\end{thm}

\begin{proof}
    We apply \autoref{thm:DistalityCriterion} with $\mathcal{N} = (\mathbb{Z};<,+)$, $\mathcal{R} = (R;<,(\cdot \equiv_n i)_{n,i})$, $\mathcal{F}_N = \left\{ +,- \right\}$, and $\mathcal{F}_R = \left\{ S,S^{-1} \right\}$.
    Condition (2) of the theorem is satisfied by \autoref{lem:ExtendabilityforR}, condition (3) by \autoref{lem:ExtendabilityforAll}, and variable separation by \autoref{lem:TermSeparationforZR}.
    Since both $\mathcal{N}$ and $\mathcal{R}$ are distal, $(\mathbb{Z};<,+,R)$ is distal.
\end{proof}

Since distal structures have NIP, the dp-rank of $(\mathbb{Z};<,+,R)$ is at most $\leq \aleph_0$.
In fact, one can construct an ICT-pattern of depth $\aleph_0$, showing that the dp-rank is exactly $\aleph_0$.
The proof is left to the reader.
One may also derive this fact from Dolich and Goodrick \cite[Corollary 2.20]{Dolich2017}.

\section{\texorpdfstring{$(\mathbb{Q}_p;+,\cdot,p^\mathbb{Z})$}{(Qp;+,×,pZ)} and \texorpdfstring{$(\mathbb{Q}_p;+,\cdot,p^\mathbb{Z},p^R)$}{(Qp;+,×,pZ,pR)}}
\label{chp:QR}

In this chapter, we prove that the structures $(\mathbb{Q}_p;+,\cdot,p^\mathbb{Z})$ and $(\mathbb{Q}_p;+,\cdot,p^\mathbb{Z},p^R)$ are distal, where $R$ is an almost sparse sequence.

\subsection{Facts and brief history}
\label{sec:QRBasics}

Recall that the \emph{$p$-adic field} $\mathbb{Q}_p$ is the field consisting of formal Laurent series of the form:
\begin{equation*}
    \mathbb{Q}_p = \left\{ \sum_{i=m}^\infty a_i p^i\;\middle|\;m \in \mathbb{Z}, a_i \in \left\{ 0,1,\dots,p-1 \right\} \right\}
\end{equation*}
equipped with canonical addition and multiplication with carries.
The $p$-adic valuation $v_p: \mathbb{Q}_p \to \mathbb{Z}$ is defined by $v_p\left( \sum_{i=m}^\infty a_i p^i \right) = m$, where $a_m$ is the first nonzero coefficient.
$\mathbb{Q}_p$ is a henselian valued field, which implies that the $p$-adic valuation is $\emptyset$-definable in the language of fields.
Consequently, we may use the valuation as if it were in the language.
For a detailed account of model theory of the $p$-adic fields, see Marker's lecture notes \cite{Marker2019}.

\begin{ntt}
    For $n \geq 2$, let $P_n$ denote the set of $n$-th powers in $\mathbb{Q}_p^\times$: $P_n = \left\{ x \in \mathbb{Q}_p^\times \mid \exists y \in \mathbb{Q}_p^\times (x = y^n) \right\}$.
    Note that $P_n$ is a subgroup of $\mathbb{Q}_p^\times$.
\end{ntt}

The following fact is a standard consequence of henselianity (cf. Marker (ibid., Chapter 2)).

\begin{fct}\label{fct:PropertiesOfPn}
    For $n \geq 2$, the quotient group $\mathbb{Q}_p^\times / P_n$ is finite, and each coset is represented by a positive integer $k$.
    Moreover, for any $n$, there exists a positive integer $m$ such that if $x \in \mathbb{Q}_p^\times$ and $\varepsilon \in \mathbb{Q}_p$ satisfy $v_p(\varepsilon) \geq v_p(x) + m$, then $(x + \varepsilon) / P^n = x / P^n$. 
\end{fct}

We later use a well-known property of general valued fields:

\begin{lem}\label{lem:ValueofAlgebraicElement}
    Let $(L,v)$ be a valued field and $K \subseteq L$ is a subfield.
    If $a \in L$ is algebraic over $K$, then there exist $n \geq 1$ and $b \in K$ such that $nv(a) = v(b)$.
\end{lem}

\begin{proof}
    Choose $f(X) = \sum_i c_i X^i \neq 0$ satisfying $f(a) = 0$, where $c_i \in K$.
    If there is only one index $i$ such that $v(c_i a^i) = \min \left\{ v(c_k a^k) \right\}_k$, then $v(f(a)) = v(c_i a^i)$, which contradicts $f(a) = 0$.
    Hence, there must exist $i < j$ such that $v(c_i a^i) = v(c_j a^j)$.
    Therefore, $(j - i) v(a) = v\left( \frac{c_i}{c_j} \right)$.
\end{proof}

We now consider the set $p^{\mathbb{Z}} = \left\{ p^z \;\middle|\; z \in \mathbb{Z} \right\}$.
Macintyre \cite{Macintyre1976} showed that $(\mathbb{Q}_p;+,\cdot)$ admits quantifier elimination in the language $\left\{ 0,1,+,-,\cdot,(P_{n})_{n\geq 2} \right\}$.
Mariaule \cite{Mariaule2017} later established quantifier elimination for $(\mathbb{Q}_p;+,\cdot,p^{\mathbb{Z}})$ by introducing a projection map.

\begin{dfn}\label{dfn:Pi}
    We define the map $\pi: \mathbb{Q}_p \to p^\mathbb{Z}$ as:
    \begin{equation*}
        \pi(a) = \begin{cases}
            p^{v_p(a)} & (a \neq 0)\\
            1 & (a = 0)
        \end{cases}.
    \end{equation*}
\end{dfn}

Note that $\pi$ is a multiplicative endomorphism on $\mathbb{Q}_p^\times$.
Also, for any $a \in \mathbb{Q}_p^\times$, $\pi(a)$ is a unique element such that $\pi(a) \in p^\mathbb{Z}$ and $v_p(\pi(a)) = v_p(a)$.
Hence, $\pi$ is $\emptyset$-definable in $(\mathbb{Q}_p;+,\cdot,p^{\mathbb{Z}})$.
Conversely, $p^\mathbb{Z}$ is defined using $\pi$ as the set $\left\{ a \in \mathbb{Q}_p \;\middle|\; \pi(a) = a \right\}$.
Therefore, $(\mathbb{Q}_p;+,\cdot,p^{\mathbb{Z}})$ and $(\mathbb{Q}_p;+,\cdot,\pi)$ coincide with respect to definable set.

\begin{thm}[{Mariaule \cite{Mariaule2017}}]\label{thm:QEforQpZ}
    $(\mathbb{Q}_p;+,\cdot,p^{\mathbb{Z}})$ admits quantifier elimination in the language $\left\{ 0,1,+,-,\cdot,(P_{n})_{n\geq 2},\pi \right\}$.
\end{thm}

In addition to quantifier elimination, these structures exhibit stability-theoretic tameness:

\begin{thm}[{Dolich, Goodrick, and Lippel \cite{Dolich2011}}]\label{thm:QpDpminimal}
    $(\mathbb{Q}_p;+,\cdot)$ is dp-minimal.
\end{thm}

\begin{thm}[{Mariaule \cite{Mariaule2018}}]\label{thm:NIPofQpwithpZ}
    $(\mathbb{Q}_p;+,\cdot,p^\mathbb{Z})$ has NIP.
\end{thm}

In fact, Mariaule gives a stronger result that its expansion by all analytic functions on $\mathbb{Z}_p$ has NIP as well.

For an almost sparse sequence $R$, we let $p^R = \left\{ p^r \mid r \in R \right\}$.
The main result of this chapter is the distality of $(\mathbb{Q}_p;+,\cdot,p^\mathbb{Z})$ and $(\mathbb{Q}_p;+,\cdot,p^\mathbb{Z},p^R)$.
The former not only strengthens Mariaule's result, but also provides a $p$-adic analogy to the distality of $(\mathbb{R};+,\cdot,2^\mathbb{Z})$ shown by Hieronymi and Nell \cite{Hieronymi2017}.
$(\mathbb{Q}_p;+,\cdot,p^\mathbb{Z},p^R)$ was not previously known to posseess even NIP; as we shall see, it provides an example of an NIP expansion of the $p$-adic field without the rationality of the Poincar\'e series.

The remainder of this chapter is organized as follows.
In \autoref{sec:QRQZ}, we establish two lemmas corresponding to conditions (2) and (3) of \autoref{thm:DistalityCriterion} for $(\mathbb{Q}_p;+,\cdot,p^\mathbb{Z})$.
\autoref{sec:QRBaF} is devoted to verifying these conditions for $(\mathbb{Q}_p;+,\cdot,p^\mathbb{Z},p^R)$ via back-and-forth argument.
Variable separation for both $(\mathbb{Q}_p;+,\cdot,p^\mathbb{Z})$ and $(\mathbb{Q}_p;+,\cdot,p^\mathbb{Z},p^R)$ is established in \autoref{sec:QRTerm}.
These results culminate in the proof of distality in \autoref{sec:QRDistal}, followed by a discussion on the non-rationality of the Poincar\'e series for $(\mathbb{Q}_p;+,\cdot,p^\mathbb{Z},p^R)$.
\subsection{Elementary behavior of \texorpdfstring{$(\mathbb{Q}_p;+,\cdot,p^\mathbb{Z})$}{(Qp;+,×,pZ)}}
\label{sec:QRQZ}

Based on Mariaule's quantifier elimination result, we show conditions (2) and (3) of \autoref{thm:DistalityCriterion} for $(\mathbb{Q}_p;+,\cdot,p^\mathbb{Z})$.

\begin{ntt}
    $a <_v b$ if and only if $v_p(a) < v_p(b)$ for $a, b \in \mathbb{Q}_p$.
\end{ntt}

Note that the relation $<_v$ is $\emptyset$-definable in $(\mathbb{Q}_p;+,\cdot)$ because the $p$-adic valuation is $\emptyset$-definable.
There is an isomorphism $(\mathbb{Z};<,+) \cong (p^\mathbb{Z};<_v,\cdot)$ given by $z \mapsto p^z$.
Let $\mathcal{U}$ be a monster model of $(\mathbb{Q}_p;+,\cdot,p^\mathbb{Z})$.

\begin{dfn}
    Let $K$ be a subfield of $U$.
    For a polynomial $P(X_0,\dots,X_{d-1})$ over $K$ and $b_0,\dots,b_{d-1} \in U$, we say that $a_I b_0^{m_0} \dots b_{d-1}^{m_{d-1}}$ is the \emph{dominant term} of $P(b_0,\dots,b_{d-1})$ if:
    \begin{equation*}
        v_p(a_I b_0^{m_0} \dots b_{d-1}^{m_{d-1}}) < v_p(a_{I'} b_0^{n_0} \dots b_{d-1}^{n_{d-1}}) - l
    \end{equation*}
    for any other term $a_{I'} b_0^{n_0} \dots b_{d-1}^{n_{d-1}}$ and any positive integer $l$.
\end{dfn}

\begin{lem}\label{lem:DominantTermProperties}
    In the setting above, we have:
    \begin{itemize}
        \item $v_p(P(b_0,\dots,b_{d-1})) = v_p(a_I b_0^{m_0} \dots b_{d-1}^{m_{d-1}})$.
        \item $\pi(P(b_0,\dots,b_{d-1})) = \pi(a_I b_0^{m_0} \dots b_{d-1}^{m_{d-1}})$.
        \item $P(b_0,\dots,b_{d_1}) / P_n = a_I b_0^{m_0} \dots b_{d-1}^{m_{d-1}} / P_n$ for any $n \geq 2$.
    \end{itemize}
\end{lem}

\begin{proof}
    The first claim is easy.
    The second follows from the fact that $v_p(x) = v_p(y) \Rightarrow \pi(x) = \pi(y)$.
    The third is a consequence of \autoref{fct:PropertiesOfPn}.
\end{proof}

The next lemma is an immediate corollary of \cite[Theorem 6.4]{Mariaule2018}.
Nevertheless, we provide a proof for the reader's convenience.
It also demonstrates a method how to deal with polynomials relying on their dominant term, which is used later again.

\begin{lem}\label{lem:ElementaryforpZ1}
    Let $K \subseteq U$ be a subfield such that $\pi(K) \subseteq K$, and let $b,c \in (p^\mathbb{Z})^\mathcal{U}$.
    If
    \begin{equation}\label{eq:TypesinpZ}\tag{*}
        \mathrm{tp}^{((p^\mathbb{Z})^\mathcal{U};<_v,\cdot)}(b / K \cap (p^\mathbb{Z})^\mathcal{U}) = \mathrm{tp}^{((p^\mathbb{Z})^\mathcal{U};<_v,\cdot)}(c / K \cap (p^\mathbb{Z})^\mathcal{U}),
    \end{equation}
    then $\mathrm{tp}^\mathcal{U}(b / K) = \mathrm{tp}^\mathcal{U}(c / K)$.
\end{lem}

\begin{proof}
    We may assume that $b,c \notin K$.

    \setcounter{clm}{0}

    \begin{clm}
        We may assume that $K$ is algebraically closed in $U$, and thus that $b$ and $c$ are transcendental over $K$.
    \end{clm}

    \begin{prooc}
        Let $K^{ac}$ be the algebraic closure of $K$ in $U$.
        For any $s \in K^{ac}$, there exist $t \in K$ and $n \geq 1$ such that $nv_p(s) = v_p(t)$ by \autoref{lem:ValueofAlgebraicElement}.
        Since $\pi$ is a multiplicative homomorphism such that $v_p(x) = v_p(y) \Rightarrow \pi(x) = \pi(y)$, It follows that $\pi(s)^n = \pi(t) \in K$.
        Thus, $K^{ac}$ is closed under $\pi$.
        Moreover, this shows that $K^{ac} \cap (p^\mathbb{Z})^\mathcal{U}$ is contained in the definable closure of $K \cap (p^\mathbb{Z})^\mathcal{U}$ within $((p^\mathbb{Z})^\mathcal{U};<_v,\cdot)$.
        Therefore, it follows:
        \begin{equation*}
            \mathrm{tp}^{((p^\mathbb{Z})^\mathcal{U};<_v,\cdot)}(b / K^{ac} \cap (p^\mathbb{Z})^\mathcal{U}) = \mathrm{tp}^{((p^\mathbb{Z})^\mathcal{U};<_v,\cdot)}(c / K^{ac} \cap (p^\mathbb{Z})^\mathcal{U}).
        \end{equation*}
        Thus, we may replace $K$ with $K^{ac}$.
    \end{prooc}

    \begin{clm}
        For any nonzero polynomial $P(X) = \sum_{i=0}^d a_i X^i \in K[X]$, $P(b)$ has a dominant term $a_{i_0} b^{i_0}$.
        Furthermore, $a_{i_0} c^{i_0}$ is the dominant term of $P(c)$.
    \end{clm}
    
    \begin{prooc}
        It is impossible for $i < j$ and $z \in \mathbb{Z}$ to satisfy $v_p(a_i b^i) = v_p(a_j b^j) + z < \infty$, since otherwise $b^{j-i} = \pi(b)^{j-i} = \frac{\pi(a_i)}{\pi(p^z a_j)} \in K$, which contradicts the transcendence of $b$.
        Thus, $P(b)$ has a dominant term $a_{i_0} b^{i_0}$.
        For any other term $a_i b^i$ and any $l$, we have:
        \begin{align*}
            v_p(a_{i_0} b^{i_0}) < v_p(a_i b^i) - l &\iff \pi(a_{i_0}) b^{i_0} <_v \pi(a_i) b^i p^{-l}\\
            &\iff \pi(a_{i_0}) c^{i_0} <_v \pi(a_i) c^i p^{-l}\\
            &\iff v_p(a_{i_0} c^{i_0}) < v_p(a_i c^i) - l,
        \end{align*}
        where the second equivalence is by \eqref{eq:TypesinpZ}.
        Hence, $a_{i_0} c^{i_0}$ is the dominant term of $P(c)$.
    \end{prooc}

    \begin{clm}
        Let $f:K[b] \to K[c]$ be the ring isomorphism sending $b$ to $c$.
        Then $K[b]$ is closed under $\pi$, and $f$ preserves $P_n$ for all $n \geq 2$ and commutes with $\pi$.
        Thus, $f$ is a partial $\left\{ +,\cdot,p^\mathbb{Z} \right\}$-elementary embedding by \autoref{thm:QEforQpZ}.
    \end{clm}

    \begin{prooc}
        For any nonzero $P(b) \in K[b]$, let $a_{i_0} b^{i_0}$ be its dominant term.
        By \autoref{lem:DominantTermProperties}, $\pi(P(b)) = \pi(a_{i_0}) b^{i_0} \in K[b]$, and $f(\pi(P(b))) = f(\pi(a_{i_0}) b^{i_0}) = \pi(a_{i_0}) c^{i_0} = \pi(P(c)) = \pi(f(P(b)))$.
        We furthermore prove that for any $n \geq 2$, $P(b) / P_n = f(P(b)) / P_n$, or equivalently $a_{i_0} b^{i_0} / P_n = a_{i_0} c^{i_0} / P_n$.
        By \eqref{eq:TypesinpZ}, for each $n$, there exists some $0 \leq i < n$ such that both $b p^{-i}$ and $c p^{-i}$ are $n$-th powers of some elements in $(p^\mathbb{Z})^\mathcal{U}$.
        $b p^{-i}$ and $c p^{-i}$ are then in $P_n$, which implies $b / P_n = c / P_n$.
    \end{prooc}
\end{proof}

\begin{lem}\label{lem:ElementaryforpZ2}
    Let $K \subseteq U$ be a subfield such that $\pi(K) \subseteq K$, and let $b,c \in U$.
    If
    \begin{equation*}
        \mathrm{tp}^{(U;+,\cdot)}(b / K) = \mathrm{tp}^{(U;+,\cdot)}(c / K)
    \end{equation*}
    and $\pi(x) \in K$ for all $x \in K(b)$, then $\mathrm{tp}^\mathcal{U}(b / K) = \mathrm{tp}^\mathcal{U}(c / K)$.
\end{lem}

\begin{proof}
    The assumption provides an field isomorphism $f:K(b) \to K(c)$ over $K$ sending $b$ to $c$, that is elementary in $(U;+,\cdot)$.
    $f$ preserves $P_n$ for all $n \geq 2$ since it is $\emptyset$-definable in the field language.
    For any $x \in K(b)$, we have $v_p(x) = v_p(\pi(x))$, and hence $v_p(f(x)) = v_p(f(\pi(x)))$ ($v_p$ is also $\emptyset$-definable in the field language).
    Since $\pi(x) \in K$, it follows $f(\pi(x)) = \pi(x)$.
    These implies $\pi(f(x)) = \pi(x) = f(\pi(x))$.
    Thus, $f$ is elementary in the language $\left\{ +,\cdot,p^\mathbb{Z} \right\}$ by \autoref{thm:QEforQpZ}.
\end{proof}
\subsection{Back-and-forth for \texorpdfstring{$(\mathbb{Q}_p;+,\cdot,p^\mathbb{Z},p^R)$}{(Qp;+,×,pZ,pR)}}
\label{sec:QRBaF}

Let $p^R = \left\{ p^r \in \mathbb{Q}_p \mid r \in R \right\}$.
In this section, we establish a back-and-forth system for the structure $(\mathbb{Q}_p;+,\cdot,p^\mathbb{Z},p^R)$.
We extend the functions $\lambda$, $S$, and $S^{-1}$ to $\mathbb{Q}_p$ as follows:

\begin{dfn}
    We define $\lambda$, $S$, and $S^{-1}$ on $p^\mathbb{Z}$ such that $(p^\mathbb{Z};<_v,\cdot,\lambda,S,S^{-1})$ is isomorphic to $(\mathbb{Z};<,+,\lambda,S,S^{-1})$.
    For $x \in \mathbb{Q}_p \setminus p^\mathbb{Z}$, we set $\lambda(x) = \lambda(\pi(x))$, $S(x) = S(\pi(x))$, and $S^{-1}(x) = S^{-1}(\pi(x))$.
    These functions are $\emptyset$-definable in $(\mathbb{Q}_p;+,\cdot,p^\mathbb{Z},p^R)$.
\end{dfn}

Let $\mathcal{U}$ be a monster model of $(\mathbb{Q}_p;+,\cdot,p^\mathbb{Z},p^R)$.
Then $((p^\mathbb{Z})^\mathcal{U};<_v,\cdot,(p^R)^\mathcal{U})$ is a monster model of $(\mathbb{Z};<,+,R)$.
Recall that an operator is a term of the form $A(x) = \sum_{i=-n}^m z_i S^i(x)$ sending $r \in R$ to $A(r) \in \mathbb{Z}$.
Since we adopt multiplicative notation in $(p^\mathbb{Z};<_v,\cdot,p^R)$, we write $A^*(x) = \prod_{i=-n}^m S^i(x)^{z_i}$ for an operator $A$.
To utilize the dominant term method, we introduce the notion of \emph{degree}.

\begin{dfn}
    The \emph{degree} of $R$, denoted by $\mathrm{deg}(R)$, is the smallest integer $d$ such that there exists an operator $A(x) = \sum_{i=0}^d z_i S^i(x)$ with $z_i \neq 0$ for some $0 \leq i \leq d$ and $A =_{ae} 0$.
    (Note that $z_0$ and $z_d$ must be nonzero because of the minimality of $d$.)
    If no such $d$ exists, we set $\mathrm{deg}(R) = \infty$.
\end{dfn}

As noted in \autoref{sec:ZRIntro}, there exists $1 < \theta \leq \infty$ such that $\lim_{n\to \infty} \frac{r_{n+1}}{r_n} = \theta$.
It follows from Lambotte and Point \cite[Lemma 1.10 to Lemma 1.13]{Lambotte2020} that $\mathrm{deg}(R) = d < \infty$ if and only if $\theta$ is an algebraic number of degree $d$.
We do not give a proof as it is not relevant for this paper.

We fix $d = \mathrm{deg}(R)$ if $\mathrm{deg}(R)$ is finite, or an arbitrary positive integer if otherwise.

\begin{lem}\label{lem:DominantTerm}
    Let $K \subseteq U$ be a subfield closed under $\pi$, $\lambda$, $S$, and $S^{-1}$, and let $b \in (p^R)^\mathcal{U} \setminus K$.
    Then for any nonzero polynomial $P \in K[X_0,\dots,X_{d-1}]$, $P(b,\dots,S^{d-1}(b))$ has a dominant term.
    In particular, $b,\dots,S^{d-1}(b)$ are algebraically independent over $K$.
\end{lem}

\begin{proof}
    Assume for contradiction that two distinct nonzero terms $a_I X_0^{m_0} \dots X_{d-1}^{m_{d-1}}$ and $a_{I'} X_0^{n_0} \dots X_{d-1}^{n_{d-1}}$ satisfy
    \begin{equation*}
        v_p(a_I b^{m_0} \dots S^{d-1}(b)^{m_{d-1}}) = v_p(a_{I'} b^{n_0} \dots S^{d-1}(b)^{n_{d-1}}) + l
    \end{equation*}
    for some $l \in \mathbb{Z}$.
    Then it holds that
    \begin{gather*}
        \pi(a_I) b^{m_0} \dots S^{d-1}(b)^{m_{d-1}} = \pi(a_{I'}) b^{n_0} \dots S^{d-1}(b)^{n_{d-1}} p^l,\\
        \text{and hence that }b^{m_0-n_0} \dots S^{d-1}(b)^{m_{d-1}-n_{d-1}} = \pi(a_I^{-1} a_{I'}) p^l.
    \end{gather*}
    Let $A(x) = \sum_{i=0}^{d-1} (m_i-n_i) S^i(x)$.
    By the definition of degree, $A \neq_{ae} 0$, so $A^*(b) \approx b$ or $A^*(b)^{-1} \approx b$ by \autoref{lem:ApproxAndLL}.
    However, $A^*(b) = \pi(a_I^{-1} a_{I'}) p^l \in K \cap (p^\mathbb{Z})^\mathcal{U}$, contradicting \autoref{lem:ElementInROutsideOfM} (1).
\end{proof}

In contrast, for any $z \in \mathbb{Z}$, $S^z(b)$ is algebraic over $K(b,\dots,S^{d-1}(b))$:

\begin{lem}\label{lem:DisEnough}
    If $d = \mathrm{deg}(R)$ is finite, then for each $z \in \mathbb{Z}$, there exist $m \geq 1$ and $w_i \in \mathbb{Z}$ for $0 \leq i < d$ such that $m S^z(x) =_{ae} \sum_{i=0}^{d-1} w_i S^i(x)$.
\end{lem}

\begin{proof}
    Choose an operator $A(x) = \sum_{i=0}^d z_i S^i(x)$ with $z_0, z_d \neq 0$ and $A =_{ae} 0$.
    Then we have $z_d S^{d+k}(x) =_{ae} \sum_{i=0}^{d-1} -z_i S^{i+k}(x)$ for all $k \geq 0$.
    Thus, we may find desired $m$ and $w_i$ inductively for $d,d+1,\dots$, similarly for $-1,-2,\dots$.
\end{proof}

We now define our back-and-forth system.

\begin{dfn}\label{dfn:BackandForthSystemforQp}
    Let $L_{p^R} = \left\{ 0,1,+,-,\cdot,{}^{-1},p^\mathbb{Z},\pi,p^R,\lambda,S,S^{-1} \right\}$.
    $\mathcal{I}$ is the set of partial functions $f:\mathcal{U} \to \mathcal{U}$ such that:
    \begin{enumerate}
        \item $\mathrm{Dom}(f)$ is an $L_{p^R}$-substructure.
        \item $f$ is a partial $L_{p^R}$-embedding.
        \item $f$ is a partial elementary embedding of $(\mathbb{Q}_p;+,\cdot,p^\mathbb{Z})$.
        \item $f|_{\mathrm{Dom}(f)\cap (p^\mathbb{Z})^\mathcal{U}}$ is a partial elementary embedding of $\left((p^\mathbb{Z})^\mathcal{U};<_v,\cdot,(p^R)^\mathcal{U}\right)$.
    \end{enumerate}

    We also consider the following alternative conditions:
    \begin{enumerate}
        \item[2'] $a \in (p^R)^{\mathcal{U}} \Rightarrow f(a) \in (p^R)^{\mathcal{U}}$ for any $a \in \mathrm{Dom}(f)$.
        \item[4'] $f|_{\mathrm{Dom}(f)\cap (p^R)^\mathcal{U}}$ is a partial elementary embedding of $\left((p^R)^\mathcal{U};<_v,(\cdot \equiv_n i)_{n,i}\right)$.
    \end{enumerate}
\end{dfn}

\begin{lem}\label{lem:ConditionstobeinIforpR}
    If a partial function $f:\mathcal{U} \to \mathcal{U}$ satisfies conditions (1), (2'), (3), and (4'), then $f$ satisfies conditions (2) and (4), and hence $f \in \mathcal{I}$.
\end{lem}

\begin{proof}
    To establish (4), we observe that $f|_{\mathrm{Dom}(f)\cap (p^\mathbb{Z})^\mathcal{U}}$ belongs to the back-and-forth system defined in \autoref{dfn:BackandForthSystemforZ} and apply \autoref{thm:PEEinR}. It is straightforward.
    By (4), $f|_{\mathrm{Dom}(f)\cap (p^\mathbb{Z})^\mathcal{U}}$ preserves $p^R$ and commutes with $\lambda$, $S$, and $S^{-1}$.
    Since $p^R \subseteq p^\mathbb{Z}$ and the functions $\lambda$, $S$, and $S^{-1}$ on $U$ are defined via $\pi$ (e.g., $\lambda = \lambda \circ \pi$), the map $f$ also satisfies these conditions on its entire domain.
    Thus, (2) is satisfied.
\end{proof}

Since $f^{-1} \in \mathcal{I}$ if $f \in \mathcal{I}$, it suffices to establish the `forth' direction.
As in \autoref{sec:ZRBaF}, we first treat the case where the added element is in $(p^R)^\mathcal{U}$.

\begin{lem}\label{lem:ExtendabilityforpR}
    Let $f \in \mathcal{I}$ with $\mathrm{Dom}(f)$ denoted by $K$.
    For any $b \in (p^R)^\mathcal{U} \setminus K$, there exists $c \in (p^R)^\mathcal{U}$ such that
    \begin{equation*}
        \mathrm{tp}^{((p^R)^\mathcal{U};<_v,(\cdot \equiv_n i)_{n,i})}(b, K\cap(p^R)^\mathcal{U}) = \mathrm{tp}^{((p^R)^\mathcal{U};<_v,(\cdot \equiv_n i)_{n,i})}(c, f(K)\cap(p^R)^\mathcal{U}).
    \end{equation*}
    Furthermore, for any such $c$, there exists an extension $\widetilde{f} \in \mathcal{I}$ of $f$ such that $\widetilde{f}(b) = c$.
\end{lem}

\begin{proof}
    Since $\mathcal{U}$ is a monster model, by \autoref{dfn:BackandForthSystemforQp} (4), there exists such $c \in (p^R)^\mathcal{U}$.
    \autoref{lem:ExtendabilityforR} and \autoref{thm:PEEinR} then imply:
    \begin{equation}\label{eq:TypeinpZ}\tag{**}
        \mathrm{tp}^{((p^\mathbb{Z})^\mathcal{U};<_v,\cdot,(p^R)^\mathcal{U})}(b, K\cap(p^\mathbb{Z})^\mathcal{U})
        = \mathrm{tp}^{((p^\mathbb{Z})^\mathcal{U};<_v,\cdot,(p^R)^\mathcal{U})}(c, f(K)\cap(p^\mathbb{Z})^\mathcal{U}).
    \end{equation}

    By \autoref{lem:DominantTerm}, $b,\dots,S^{d-1}(b)$ (resp. $c,\dots,S^{d-1}(c)$) are algebraically independent over $K$ (resp. $f(K)$).
    We extend $f$ to a ring isomorphism $f_1:K[b,\dots,S^{d-1}(b)] \to f(K)[c,\dots,S^{d-1}(c)]$ by defining $f_1(S^i(b)) = S^i(c)$ for each $0 \leq i \leq d-1$ (if $\mathrm{deg}(R) = \infty$, then $\mathrm{Dom}(f_1) = K[b,S(b),S^2(b),\dots]$).

    \setcounter{clm}{0}

    \begin{clm}
        $K[b,\dots,S^{d-1}(b)]$ is closed under $\pi$.
        Moreover, $f_1$ preserves $P_n$ for all $n \geq 2$ and commutes with $\pi$.
        Consequently, $f_1$ is a partial $\left\{ +,\cdot,p^\mathbb{Z} \right\}$-elementary embedding by \autoref{thm:QEforQpZ}.
    \end{clm}

    \begin{prooc}
        Consider any nonzero $a = P(b,\dots,S^{d-1}(b)) \in K[b,\dots,S^{d-1}(b)]$.
        Let $a_I b^{m_0} \dots S^{d-1}(b)^{m_{d-1}}$ be the dominant term of $a$, which exists by \autoref{lem:DominantTerm}.
        By arguing similarly to Claim 2 in \autoref{lem:ElementaryforpZ1} with \eqref{eq:TypeinpZ}, it follows that $f(a_I) c^{m_0} \dots S^{d-1}(c)^{m_{d-1}}$ is the dominant term of $f_1(a)$.
        Also, for any $n \geq 2$, there exists $k \in \mathbb{Z}$ such that $a_I / P_n = m / P_n$ by \autoref{fct:PropertiesOfPn}, and hence $f(a_I) / P_n = k / P_n = a_I / P_n$.
        As in Claim 3 of \autoref{lem:ElementaryforpZ1}, we can show that $K[b,\dots,S^{d-1}(b)]$ is closed under $\pi$, $f_1$ commutes with $\pi$, and that $f_1$ preserves $P_n$ for all $n \geq 2$.
    \end{prooc}
    
    By Claim 1, we extend $f_1$ to a partial $\left\{ +,\cdot,p^\mathbb{Z} \right\}$-elementary embedding $\widetilde{f}:K(b,\dots,S^{d-1}(b))^{ac} \to \mathcal{U}$ (algebraic closure is taken within $U$).

    \begin{clm}
        $\widetilde{f} \in \mathcal{I}$.
    \end{clm}

    \begin{prooc}
        \autoref{dfn:BackandForthSystemforQp} (3) is the definition of $\widetilde{f}$.
        We first show \autoref{dfn:BackandForthSystemforQp} (1); that is, $K(b,\dots,S^{d-1}(b))^{ac}$ is closed under $\pi,\lambda$, $S$, and $S^{-1}$.
        For any nonzero $a \in K(b,\dots,S^{d-1}(b))^{ac}$, there exists $n \geq 1$ and $x,y \in K[b,\dots,S^{d-1}(b)]$ such that $nv_p(a) = v_p(x / y)$ by Lemma \ref{lem:ValueofAlgebraicElement}.
        By \autoref{lem:DominantTermProperties}, $\pi(a)^n = \pi(x / y) = \pi(x) / \pi(y) = a_0 b^{z_0} \dots S^{d-1}(b)^{z_{d-1}}$ for some $a_0 \in K \cap (p^\mathbb{Z})^\mathcal{U}$ and $z_i \in \mathbb{Z}$.
        Thus, $\pi(a) \in K(b,\dots,S^{d-1}(b))^{ac}$.
        To show the assertion for $\lambda$, $S$, and $S^{-1}$, we may assume $a \in (p^\mathbb{Z})^\mathcal{U}$ and that $a >_v 1$.
        Then $a^n = a_0 b^{z_0} \dots S^{d-1}(b)^{z_{d-1}}$.
        By \autoref{lem:ApproxAndLL} (1), we have $a \approx a^n$.
        By \autoref{lem:ElementInROutsideOfM} (2), $a_0 b^{z_0} \dots S^{d-1}(b)^{z_{d-1}}$ approximately equals $a_0$ or $b$.
        Hence, either $a \approx a_0$ or $a \approx b$ holds.
        We see, for some $z$, that $\lambda(a)$ equals to either $S^z(a_0)$ or $S^z(b)$, both of which are in $K(b,\dots,S^{d-1}(b))^{ac}$ by \autoref{lem:DisEnough}.

        Next, we verify \autoref{dfn:BackandForthSystemforQp} (2') and (4').
        For any $a \in K(b,\dots,S^{d-1}(b))^{ac} \cap (p^R)^\mathcal{U}$, the argument above shows that $a$ is either in $K \cap (p^R)^\mathcal{U}$ or is of the form $S^z(b)$ for some $z \in \mathbb{Z}$.
        By \autoref{lem:DisEnough}, for any $z \in \mathbb{Z}$, there exists $m \geq 1$ $\widetilde{f}((S^z(b))^m) = (S^z(c))^m$ and hence $\widetilde{f}(S^z(b)) = S^z(c)$ (since $(p^\mathbb{Z})^\mathcal{U}$ is torsion-free).
        conditions (2') and (4') are satisfied by the same reasoning as Claim 2 in \autoref{lem:ExtendabilityforR}.
    \end{prooc}
\end{proof}

\begin{lem}\label{lem:SubLemmaforExtendabilityforAllinpR}
    For any $b \in \mathcal{U}$ and any $f \in \mathcal{I}$, there exists an extension $\widetilde{f} \in \mathcal{I}$ of $f$ such that for any $x \in \langle b, \mathrm{Dom}(\widetilde{f}) \rangle_{\left\{ +,-,\cdot,{}^{-1},\pi \right\}}$, we have $\lambda(x) \in \mathrm{Dom}(\widetilde{f})$.
\end{lem}

\begin{proof}
    Let $f_0 = f$.
    We construct an increasing sequence $(f_n)_{n\in\mathbb{N}}$ in $\mathcal{I}$ as follows:
    Suppose $f_n \in \mathcal{I}$ has been defined.
    We then extend $f_n$ to $f_{n+1} \in \mathcal{I}$ such that $\mathrm{Dom}(f_{n+1})$ contains all elements of $\langle b, \mathrm{Dom}(f_n) \rangle_{\left\{ +,-,\cdot,{}^{-1},\pi,\lambda,S,S^{-1} \right\}} \cap (p^R)^\mathcal{U}$.
    This is achieved by applying \autoref{lem:ExtendabilityforpR} transfinitely many times.

    We define $\widetilde{f} = \bigcup_{n<\omega} f_n \in \mathcal{I}$.
    For any $x \in \langle b, \mathrm{Dom}(\widetilde{f}) \rangle_{\left\{ +,-,\cdot,{}^{-1},\pi \right\}}$, there exists $n \geq 0$ such that $x \in \langle b, \mathrm{Dom}(f_n) \rangle_{\left\{ +,-,\cdot,{}^{-1},\pi \right\}}$.
    By our construction, $\lambda(x) \in \mathrm{Dom}(f_{n+1}) \subseteq \mathrm{Dom}(\widetilde{f})$, as desired.
\end{proof}

\begin{lem}\label{lem:ExtendabilityforAllinpR}
    Suppose $b \in \mathcal{U}$ and $f \in \mathcal{I}$ satisfy that $\lambda(x) \in \mathrm{Dom}(f)$ for any $x \in \langle b, \mathrm{Dom}(f) \rangle_{\left\{ +,-,\cdot,{}^{-1},\pi \right\}}$.
    Then there exists $c \in \mathcal{U}$ such that
    \begin{equation*}
        \mathrm{tp}^{\left( U;+,\cdot,(p^\mathbb{Z})^\mathcal{U} \right)} (b, \mathrm{Dom}(f)) = \mathrm{tp}^{\left( U;+,\cdot,(p^\mathbb{Z})^\mathcal{U} \right)} (c, \mathrm{Im}(f)).
    \end{equation*}
    Furthermore, for any such $c$, there exists an extension $\widetilde{f} \in \mathcal{I}$ of $f$ such that $\widetilde{f}(b) = c$.
\end{lem}

\begin{proof}
    The existence of such a $c$ follows from the saturation of $\mathcal{U}$ and \autoref{dfn:BackandForthSystemforQp} (3).
    This choice of $c$ allows us to extend $f$ to a partial $\left\{ +,\cdot,p^\mathbb{Z} \right\}$-elementary embedding $\widetilde{f}:\langle b, \mathrm{Dom}(f) \rangle_{\left\{ +,-,\cdot,{}^{-1},\pi \right\}} \to \mathcal{U}$ in the natural way.
    We now verify that $\widetilde{f} \in \mathcal{I}$.
    \autoref{dfn:BackandForthSystemforQp} (3) trivially holds.
    The assumption on $b$ and $f$ ensures \autoref{dfn:BackandForthSystemforQp} (1) and $\mathrm{Dom}(\widetilde{f}) \cap (p^R)^\mathcal{U} = \mathrm{Dom}(f) \cap (p^R)^\mathcal{U}$, which implies \autoref{dfn:BackandForthSystemforQp} (2') and (4').
\end{proof}

\begin{thm}\label{thm:PEEinpR}
    $\mathcal{I}$ forms a back-and-forth system in the language $\left\{ +,\cdot,p^\mathbb{Z},p^R \right\}$.
    Consequently, every $f \in \mathcal{I}$ is a partial $\left\{ +,\cdot,p^\mathbb{Z},p^R \right\}$-elementary embedding.
\end{thm}
\subsection{Variable separation}
\label{sec:QRTerm}

In this section, we assume $\mathcal{U}$ is a monster model of either $(\mathbb{Q}_p;+,\cdot,p^\mathbb{Z})$ or $(\mathbb{Q}_p;+,\cdot,p^\mathbb{Z},p^R)$.
By `terms', we refer to $\left\{ +,-,\cdot,{}^{-1},\pi \right\}$-terms in the former case, and $\left\{ +,-,\cdot,{}^{-1},\pi,\lambda,S,S^{-1} \right\}$-terms in the latter.

Fix a sequence $(\overline{a_k})_{k\in I+(c)+J}$ and a singleton $b$ satisfying the assumptions in \autoref{dfn:Distality}, with $I$ and $J$ taken as sufficiently saturated dense linear orders without endpoints as noted in \autoref{rmk:SaturatedIndices}.
We keep following the terminology in \autoref{dfn:IndisParameter}.

\begin{lem}\label{lem:DecompositionofValuation}
    For any terms $(u_i(\overline{x}))_{1\leq i\leq m}$ and $(r_i(y))_{1\leq i\leq m}$, there exist a truncation $(I',J') \subset_{\mathrm{t}} (I,J)$ and terms $u(\overline{x})$ and $r(y)$ involving parameters outside $(I',J')$ such that, for all $k \in I'+(c)+J'$:
    \begin{equation*}
        v_p\left( \sum_{i=1}^m u_i(\overline{a_k}) r_i(b) \right) = v_p(u(\overline{a_k}) r(b)).
    \end{equation*}
\end{lem}

\begin{proof}
    We proceed by induction on $m$. The case $m = 1$ is trivial. Suppose that the claim holds for $m$.
    We may assume that $u_{m+1}(\overline{a_k})$ and $r_{m+1}(b)$ are both not $0$ for all $k \in I+(c)+J$. Moreover, since
    \begin{equation*}
        v_p\left( \sum_{i=1}^{m+1} u_i(\overline{a_k}) r_i(b) \right) = v_p(u_{m+1}(\overline{a_k}) r_{m+1}(b)) + v_p\left( \sum_{i=1}^m \frac{u_i(\overline{a_k})}{u_{m+1}(\overline{a_k})} \frac{r_i(b)}{r_{m+1}(b)} + 1 \right),
    \end{equation*}
    we may further assume that $u_{m+1}(\overline{a_k}) = r_{m+1}(b) = 1$.

    Let $\overline{u}(\overline{a_k}) = (u_1(\overline{a_k}),\dots,u_m(\overline{a_k}))$ for $k \in I+(c)+J$.
    Since $(\overline{u}(\overline{a_k}))_{k\in I+(c)+J}$ is indiscernible and $(\overline{u}(\overline{a_k}))_{k\in I+J}$ is indiscernible over $\left\{ r_1(b),\dots,r_m(b) \right\}$, the distality of the $p$-adic fields (\autoref{rmk:DistalImplications}) implies that $(\overline{u}(\overline{a_k}))_{k\in I+(c)+J}$ is indiscernible over $\left\{ r_1(b),\dots,r_m(b) \right\}$ in the language $\left\{ +,\cdot \right\}$.
    Since any three points in an ultrametric space form an isosceles triangle, this indiscernibility ensures that there are three cases to consider:
    
    \setcounter{case}{0}
    \begin{case}
        $v_p\left( \sum_{i=1}^m u_i(\overline{a_k}) r_i(b) + 1 \right) = v_p\left( \sum_{i=1}^m u_i(\overline{a_l}) r_i(b) + 1 \right)$ for all $k < l$ from $I+(c)+J$.

        Choose some $i_0 \in I$ and shrink $I$ to $I'$ such that $i_0 < I'$.
        We then take a term $r(y) = \sum_{i=1}^m u_i(\overline{a_{i_0}}) r_i(y) + 1$, which involves parameters outside $(I',J)$.
    \end{case}

    \begin{case}
        $v_p\left( \sum_{i=1}^m u_i(\overline{a_k}) r_i(b) + 1 \right) = v_p\left( \sum_{i=1}^m u_i(\overline{a_k}) r_i(b) - \sum_{i=1}^m u_i(\overline{a_l}) r_i(b) \right)$ for all $k < l$ from $I+(c)+J$.
        
        Choose some $j_0 \in J$ and shrink $J$ to $J'$ such that $J' < j_0$.
        Then for all $k \in I+(c)+J'$, we have:
        \begin{equation*}
            v_p\left( \sum_{i=1}^m u_i(\overline{a_k}) r_i(b) + 1 \right) = v_p\left( \sum_{i=1}^m \left( u_i(\overline{a_k}) - u_i(\overline{a_{j_0}}) \right) r_i(b) \right).
        \end{equation*}
        Since each $u_i(\overline{x}) - u_i(\overline{a_{j_0}})$ is a term involving parameters outside $(I,J')$, the induction hypothesis applies to the right-hand side.
    \end{case}    
    
    \begin{case}
        $v_p\left( \sum_{i=1}^m u_i(\overline{a_l}) r_i(b) + 1 \right) = v_p\left( \sum_{i=1}^m u_i(\overline{a_k}) r_i(b) - \sum_{i=1}^m u_i(\overline{a_l}) r_i(b) \right)$ for all $k < l$ from $I+(c)+J$.

        We proceed similarly to Case 2 shrinking $I$.
    \end{case}
\end{proof}

\begin{lem}\label{lem:TermSeparationpR}
    For any term $t(\overline{x},y)$, there exist a truncation $(I',J') \subset_{\mathrm{t}} (I,J)$ and terms $\alpha_i(\overline{x})$, $\beta_i(\overline{x})$, $r_i(y)$, $s_i(y)$ ($1 \leq i \leq m$) involving parameters outside $(I',J')$ such that
    \begin{equation*}
        t(\overline{a_k},b) = \dfrac{\sum_i \alpha_i(\overline{a_k}) r_i(b)}{\sum_i \beta_i(\overline{a_k}) s_i(b)}
    \end{equation*}
    for all $k \in I'+(c)+J'$.
    Furthermore, if $t$ is of the form $\pi(t')$, $\lambda(t')$, $S(t')$, or $S^{-1}(t')$, then $t(\overline{a_k},b) = \alpha(\overline{a_k}) r(b)$, and $\alpha(\overline{a_k}), r(b) \in (p^\mathbb{Z})^\mathcal{U}$.
    Moreover, either $\alpha$ or $r$ equals $1$ for $\lambda$, $S$, or $S^{-1}$.
\end{lem}

\begin{proof}
    The proof is by induction on the complexity of $t$.
    Nontrivial cases are when $t$ is of the form $\pi(t')$, $\lambda(t')$, $S(t')$, or $S^{-1}(t')$.
    The case for $\pi$ follows from \autoref{lem:DecompositionofValuation}, as $v_p(x) = v_p(y) \Rightarrow \pi(x) = \pi(y)$.
    For $\lambda$, $S$, and $S^{-1}$, we may assume that $t'(\overline{a_k},b) = \alpha(\overline{a_k}) r(b)$ and $\alpha(\overline{a_k}), r(b) \in (p^\mathbb{Z})^\mathcal{U}$ since e.g. $\lambda = \lambda \circ \pi$.
    We then proceed as in \autoref{lem:TermSeparationforZR} to obtain the desired form.
\end{proof}
\subsection{Distality and \texorpdfstring{Poincar\'e}{Poincaré} series}
\label{sec:QRDistal}

The $p$-adic field $(\mathbb{Q}_p;+,\cdot)$ is distal (\autoref{rmk:DistalImplications}).
Additionally, the structures $(p^\mathbb{Z};<_v,\cdot) \cong (\mathbb{Z};<,+)$ and $(p^R;<_v,(\cdot \equiv_n i)_{n,i})$ are distal as well (\autoref{thm:NIPofPA} and \autoref{thm:ColoredOrderDpminimality}).

\begin{thm}
    $(\mathbb{Q}_p;+,\cdot,p^\mathbb{Z})$ is distal.
\end{thm}

\begin{proof}
    We apply the criterion in \autoref{thm:DistalityCriterion} by setting $\mathcal{N} = (\mathbb{Q}_p;+,\cdot)$, $\mathcal{R} = (p^\mathbb{Z};<_v,\cdot)$, $\mathcal{F}_N = \left\{ +,-,\cdot,{}^{-1} \right\}$, $\mathcal{F}_R = \emptyset$, and $\lambda = \pi$.
    The required conditions are satisfied by \autoref{lem:ElementaryforpZ1}, \autoref{lem:ElementaryforpZ2}, and \autoref{lem:TermSeparationpR}.
\end{proof}

\begin{thm}
    $(\mathbb{Q}_p;+,\cdot,p^\mathbb{Z},p^R)$ is distal.
\end{thm}

\begin{proof}
    We again apply \autoref{thm:DistalityCriterion}, this time setting $\mathcal{N} = (\mathbb{Q}_p;+,\cdot,p^\mathbb{Z})$, $\mathcal{R} = (p^R;<_v,(\cdot \equiv_n i)_{n,i})$, $\mathcal{F}_N = \left\{ +,-,\cdot,{}^{-1},\pi \right\}$, and $\mathcal{F}_R = \left\{ S,S^{-1} \right\}$.
    The conditions are verified by \autoref{lem:ExtendabilityforpR}, \autoref{lem:ExtendabilityforAllinpR}, and \autoref{lem:TermSeparationpR}.
\end{proof}

As distal structures possess NIP, $(\mathbb{Q}_p;+,\cdot,p^\mathbb{Z},p^R)$ is an NIP expansion.
By constructing an ICT-pattern of depth $\aleph_0$, one can show that $(\mathbb{Q}_p;+,\cdot,p^\mathbb{Z})$ has dp-rank exactly $\aleph_0$.
We leave the proof to the reader.
The NIP of these expansion is particularly significant in the context of Poincar\'e series.

\begin{dfn}
    For a subset $A$ of $\mathbb{Z}_p^n$ and  $m \geq 0$, let $N_m$ be the cardinality of the set $\left\{ (x_1 \bmod p^m, \dots, x_n \bmod p^m) \in (\mathbb{Z} / p^m \mathbb{Z})^n \mid \overline{x} \in A \right\}$.
    The \emph{Poincar\'e series} of $A$ is the formal power series $P_A(T) = \sum_{m=0}^\infty N_m T^m$.
\end{dfn}

A celebrated result by Denef \cite{Denef1984} states that if $A \subseteq \mathbb{Z}_p^n$ is definable in the pure $p$-adic field $(\mathbb{Q}_p;+,\cdot)$, its Poincar\'e series is a rational function.
This rationality extends to any dp-minimal extension of the $p$-adic field by combining Simon and Walsberg \cite{Simon2025} (dp-minimal $\Rightarrow$ P-minimal) and Kovacsics and Leenknegt \cite{Kovacsics2016} (P-minimal $\Rightarrow$ rationality).
While Denef \cite{Denef1985} showed that $(\mathbb{Q}_p;+,\cdot,p^\mathbb{Z})$ also admits rationality despite having dp-rank $\aleph_0$, we show that this does not hold for all NIP expansions.

\begin{prop}
    The Poincar\'e series of $p^R$ is not a rational function.
\end{prop}

\begin{proof}
    For simplicity, consider $R = \left\{ 2^n \mid n \in \mathbb{N} \right\}$.
    For each $m \geq 0$, $N_m = \left| \left\{ x \bmod p^m \mid x \in p^R \right\} \right|$ equals $1 + \left| \left\{ n \mid 2^n < m \right\} \right|$.
    Thus, the Poincar\'e series of $p^R$ is given by:
    \begin{equation*}
        P(T) = \sum_{m=0}^\infty T^m + \sum_{n=0}^\infty \sum_{m=2^n+1}^\infty T^m,
    \end{equation*}
    which simplifies to $(1 - T) P(T) = 1 + \sum_{n=0}^\infty T^{2^n+1}$.
    It suffice to show that the right-hand side is not rational.
    If this were a rational function, there would exist nonzero polynomials $f(T), g(T) \in \mathbb{Z}[T]$ such that $f(T) = g(T) \left( 1 + \sum_{n=0}^\infty T^{2^n+1} \right)$.
    However, for sufficiently large $n$, $g(T) T^{2^n+1}$ and $g(T) T^{2^{n+1}+1}$ cannot overlap for $g(T) \left( 1 + \sum_{n=0}^\infty T^{2^n+1} \right)$ to form a polynomial.
\end{proof}

Consequently, $(\mathbb{Q}_p;+,\cdot,p^\mathbb{Z},p^R)$ provides an example of an NIP expansion of the $p$-adic field that lacks the rationality of the Poincar\'e series.
Since this structure has dp-rank $\aleph_0$, a natural question arises regarding structures with lower rank:

\begin{qst}\label{qst:StrongDependenceRationality}
    If an expansion of the $p$-adic field is strongly dependent (i.e., its dp-rank is less than $\aleph_0$), must it have the rationality of the Poincar\'e series?
\end{qst}

In the context of the real field, strong dependence implies the tameness of open definable sets.
To elaborate, let $\mathcal{M}$ be a structure on a topological space $M$ that defines a basis for the topology.
The open core $\mathcal{M}^\circ$ is the reduct of $\mathcal{M}$ whose predicate symbols are the open sets definable in $\mathcal{M}$.
Furthermore, an expansion $\widetilde{\mathcal{M}}$ of $\mathcal{M}$ is $\mathcal{M}$-minimal if every unary set definable in $\widetilde{\mathcal{M}}$ is already definable in $\mathcal{M}$.

\begin{thm}[Walsberg {\cite[Theorem 9.15]{Walsberg2022}}]
    If an expansion $\mathcal{R}$ of $(\mathbb{R};<,+)$ is strongly dependent, then its open core $\mathcal{R}^\circ$ is either o-minimal or $(\mathbb{R};<,+,\alpha \mathbb{Z})$-minimal for some $\alpha > 0$.
\end{thm}

In particular, if $\mathcal{R}$ is a strongly dependent expansion of the real field, then $\mathcal{R}^\circ$ is o-minimal, as it defines multiplication and thus cannot be $(\mathbb{R};<,+,\alpha \mathbb{Z})$-minimal.
We may ask if an analogous situation holds for the $p$-adic field:

\begin{qst}
    For any strongly dependent expansion of the $p$-adic field, is its open core P-minimal?
\end{qst}

A positive answer to this question would imply a positive answer to \autoref{qst:StrongDependenceRationality}.
It is easily observed that any subset $A$ of $\mathbb{Z}_p^n$ and its topological closure have the same Poincar\'e series.
By Kovacsics and Leenknegt \cite{Kovacsics2016}, the P-minimality of the open core would ensure rationality for the original expansion as well.

\addcontentsline{toc}{section}{References}
\printbibliography

\end{document}